\documentclass[11pt, article, reqno]{amsart}

\allowdisplaybreaks

\usepackage[margin=16mm]{geometry}
\usepackage{amssymb, amsmath, amsthm, amscd}

\usepackage[T1]{fontenc}
\usepackage{eucal,mathrsfs}
\usepackage{color}

\renewcommand{\le}{\leqslant}
\renewcommand{\ge}{\geqslant}

\definecolor{mno}{rgb}{0.5,0.1,0.5}

\newcommand{\R}{\mathbb R}

\newcommand{\Pp}{\mathbb P}
\newcommand{\Ee}{\mathbb E}

\newcommand{\I}{\mathbf 1}
\newcommand{\w}{\omega}
\newcommand{\N}{\mathbb{N}}
\newcommand{\Z}{\mathbb Z}
\newcommand{\sL}{\mathcal{L}}

\newcommand{\E}{\mathscr{E}}
\newcommand{\F}{\mathscr{F}}

\newcommand{\LL}{\mathcal{L}}

\newtheorem{theorem}{Theorem}[section]

\newtheorem{proposition}[theorem]{Proposition}

\theoremstyle{definition}

\numberwithin{equation}{section}

\begin{document}

\title[Quantitative stochastic homogenization for random walks with critical jump index] {Quantitative stochastic homogenization
for long-range random walks with critical jump index}

\author{Xin Chen,\quad Chenlin Gu  \quad \hbox{and}\quad Jian Wang}

\date{}

\maketitle

\begin{abstract}
In this paper, we study the stochastic homogenization for a class of symmetric random walks
in random conductance model, whose one-step transition probability from  $x$ to $y$ is proportional to $|x-y|^{-d-2}$.
As the associated jumping kernel fails to be $L^2$-integrable yet admits a finite $\alpha$-th moment for all $\alpha\in (0,2)$, we refer to the corresponding process $(X^\w_t)_{t\ge0}$ as a long-range random walk with critical jump index. In this critical regime, the scaled process
$\bigl(k^{-1}X_{k^2(\log k)^{-1}t}\bigr)_{t\ge 0}$,
whose scaling order is different from
the diffusive scaling and the $\alpha$-stable scaling,
converges to a Brownian motion.
Besides characterizing the limiting Brownian motion, we will give a convergence rate for associated
scaled resolvents, which obeys the order $(\log k)^{-\frac{1}{2}+\frac{1}{2(d-2)}+\varepsilon}$ with any $\varepsilon>0$
for all $d>3$.

\medskip

\noindent\textbf{Keywords:} stochastic homogenization;
random conductance model; long-range jumps;
critical jump index

\medskip

\noindent \textbf{MSC 2010:} 60G51; 60G52; 60J25; 60J75.
\end{abstract}
\allowdisplaybreaks

\section{Introduction and main result}\label{section1}
Stochastic homogenization studies the limiting behavior of operators or stochastic processes with random coefficients under scaling, and aims to characterize the macroscopic deterministic effective behavior from microscopic random media. It is a core interdisciplinary topic in probability theory, partial differential equations, and statistical physics.
After the pioneer work by Kozlov \cite{Ko} and  Papanicolaou and Varadhan \cite{PV}, where the so-called -- seen from particle -- method
has been introduced independently, there are many different results on the qualitative stochastic homogenization, see
\cite{ABDH,ACJS,ADS,BB,B,FK,GZ,KV,Ku, MP,SS} and reference therein for details. Recently, based on several new ideas,
such as the sub-additivity and the dual expression for the corresponding energy functional, functional inequalities
for the forms associated with vertical derivatives in configuration (environments) spaces, and the renormalization group technique
to analyze the coarse-grained diffusivity among different scales, there is much progress on the theory of quantitative stochastic homogenization, including
various applications on the scaling limit of interacting particle systems, the regularity theory of partial different equations,
and the invariance principle for critically correlated stochastic turbulence flows.
We refer the reader to
\cite{ABT,AD,AK,AKM,AS,AW,CMOW,DG,FGW,GGM,GNO,Go1,Go3} and reference therein.

In particular, the results mentioned above mainly focus on stochastic homogenization for nearest-neighbor random walks or divergence-form elliptic differential operators.
Due to the diversity and the complexity of long-range jumps which may allow jumps of arbitrary size, long-range random walks (or jump processes) usually posses completely different large scale
properties compared with nearest-neighbor random walks. These models arise naturally in the study of disordered media, nonlocal diffusion, anomalous transport, and fractional dynamics.
Recently, by introducing some methods to study the effects of long-range jumps,
there are a few developments on stochastic homogenization for long-range random walks in random media or non-local operators with random coefficients; see \cite{BCKW,CCKW2,CCKW3,CCKW2025,CKW21,CKK,FHS,KPZ,PZ}.

In this paper, we consider a long-range random walk on $\Z^d$ with the following (non-local) infinitesimal operator
\begin{equation}\label{e1-1}
\sL^\w f(x)=
\sum_{y\in \Z^d: y\not=x}
\left(f(y)-f(x)\right)\frac{w_{x,y}(\w)}{|x-y|^{d+2}},
\end{equation}
 where the random coefficients
$\{w_{x,y}(\w)\}_{x,y\in \Z^d}$ satisfy the assumption as follows:

\noindent{{\bf Assumption (H1)}}: {\it Let $E=\{(x, y):x,y\in \Z^d\}$ be the collection of all unordered
distinct pairs
on $\Z^d$.
Suppose that $\{w_{x,y}: (x,y)\in E\}$ is a sequence of independent random variables such that the following
hold{\rm:}
\begin{itemize}
\item [(i)] $w_{x,y}=w_{y,x}$ for
every $x\not= y\in \Z^d${\rm;}

\item [(ii)] $\Ee[w_{x,y}]=1$ for
every $x\not= y\in \Z^d${\rm;}

\item [(iii)] There exist positive constants $C_1\le C_2$ such that
\begin{equation}\label{a1-1-1}
C_1\le w_{x,y}(\w)\le C_2,\quad x\not=y \in \Z^d,\ \w\in \Omega.
\end{equation}
\end{itemize}}

Denote by $\mu$ the counting measure on $\Z^d$. Given  $\omega\in \Omega$,
let $(\mathscr{E}^\w, \mathscr{F}^\w)$ be
a symmetric  Dirichlet form on $L^2(\Z^d;\mu)$ defined by
\begin{align*}
\mathscr{E}^\w(f,f)=&\frac{1}{2}
\sum_{x,  y\in \Z^d: x\not= y}
\left(f(x)-f(y)\right)^2\frac{w_{x,y}(\w)}{|x-y|^{d+2}}=\frac{1}{2}\int_{\Z^d}\int_{\Z^d}
\left(f(x)-f(y)\right)^2\frac{w_{x,y}(\w)}{|x-y|^{d+2}}\,\mu(dx)\,\mu(dy),\\
\mathscr{F}^\w=&\{f\in L^2(\Z^d;\mu): \mathscr{E}^\w (f,f)<\infty\}.
\end{align*}
According to the proof of \cite[Theorem 3.2]{CKK},
the Dirichlet form $(\mathscr{E}^\w, \mathscr{F}^\w)$ is regular   on $L^2(\Z^d; \mu)$ with
core
${\mathcal B}_c(\Z^d)$,
which is the set of functions on $\Z^d$ with compact support.
Then, there exists a $\Z^d$-valued symmetric Hunt process
$(X_t^\omega)_{t\ge 0}$ associated with the regular Dirichlet form $(\mathscr{E}^\w, \mathscr{F}^\w)$.  It is easy to see that the infinitesimal generator $\sL^\w$ corresponding to
$(\mathscr{E}^\w, \mathscr{F}^\w)$ has the expression \eqref{e1-1}.
This expression  reveals that the process
$(X_t^\omega)_{t\ge 0}$ on $\Z^d$ is with the jump intensity from $x$ to $y$ being $\frac{w_{x,y}(\w)}{|x-y|^{d+2}}$.

For any $k\in \N_+:=\{1,2,\cdots\}$, let
 $\mu^{(k)}$
 be the normalized counting measure on $k^{-1}\Z^d$ so that
$\mu^{(k)} ((0, 1]^d)=1$, i.e., for any $A\subset k^{-1}\Z^d$, $\mu^{(k)}(A):=
k^{-d}\sum_{x\in k^{-1} \Z^d}\I_A(x)$.
Define
the  scaled operator
$\sL^{(k),\w}$ as follows:
\begin{equation}\label{e:operator}
\sL^{(k),\w} f(x):=(\log k)^{-1} k^{-d}
  \sum_{y\in k^{-1}\Z^d: y\not= x}
\left(f(y)-f(x)\right)\frac{w_{kx,ky}(\w)}{|x-y|^{d+2}},\quad x\in k^{-1}\Z^d,\
f\in L^2(k^{-1}\Z^d;\mu^{(k)}).
\end{equation}
In fact, it is not difficult to verify that $\sL^{(k),\w}$ is the infinitesimal generator of
the scaled process $(X_t^{(k),\w})_{t\ge 0}:=\left(k^{-1}X_{k^2 (\log k)^{-1}t}\right)_{t\ge 0}$.
In particular, the scaling order ${k^2}({\log k})^{-1}$ is neither the diffusive order $k^2$
(which was studied by \cite{CCKW2,CCKW2025,CKW21,CKK,KPZ}) nor the
$\alpha$-stable order $k^\alpha$ with $\alpha\in (0,2)$ (which was studied by \cite{CCKW2,CCKW2025,CKW21,CKK,KPZ}) in the procedure of stochastic homogenization.
That is why we call the  process $(X_t^\omega)_{t\ge 0}$ a long-range random walk with critical jump index, which can be
viewed as a transition from the $\alpha$-stable case (with $\alpha\in (0,2)$) to the diffusive case.
Moreover, since
\begin{equation}\label{e0-1}
\Ee\left[\sum_{y\in \Z^d}|x-y|^2\cdot\frac{w_{x,y}(\w)}{|x-y|^{d+2}}\right]=\infty,\quad
\Ee\left[\sum_{y\in \Z^d}|x-y|^\alpha\cdot\frac{w_{x,y}(\w)}{|x-y|^{d+2}}\right]<\infty,\,\, \alpha\in (0,2),
\end{equation}
the jumping kernel of the process $(X_t^\omega)_{t\ge 0}$ is not $L^2$-integrable, but it has finite $\alpha$-th moment for
all $\alpha\in (0,2)$, which also ensures the critical scaling order for the process $(X_t^\w)_{t\ge 0}$ in some sense.

The main objective of this paper is to study both qualitative and quantitative properties
for the stochastic homogenization of $(X_t^\w)_{t\ge 0}$, which are still unknown to the best of our knowledge.
In particular, we will establish
the $L^2$-convergence rate
for the $\lambda$-resolvent
  of $ \sL^{(k),\w}$ as $k\to \infty$ to that
of the limit
homogenized generator  $\bar \sL$, which
has the following expression
\begin{equation}\label{e:1.2}
\bar \sL f(x):= a_0\Delta f(x),\quad f\in C_c^2(\R^d).
\end{equation}
Here, $\Delta$ denotes the standard Laplacian operator on $\R^d$, and
\begin{equation}\label{e1-3a}
a_0:=\lim_{k \to \infty}\frac{1}{2d\log k}\left(\sum_{z\in \Z^d:|z|\le k}|z|^{-d}\right).
\end{equation}

For $\lambda >0$, let
$$
\mathscr{S}_0^\lambda:= \left\{f:  f= (\lambda - \bar \sL) g \hbox{ for some } g\in   C_c^\infty(\R^d) \right\}.
$$
This is,
$$
\mathscr{S}_0^\lambda =\{f:  \bar R_\lambda f \in C_c^\infty (\R^d)\},
$$ where $\bar R_\lambda$ is the $\lambda$-resolvent of the operator $\bar \sL$.
Then,
$\mathscr{S}_0^\lambda \subset C_b(\R^d)\cap L^2(\R^d;dx)\cap L^2(k^{-1}\Z^d;\mu^{(k)})$ for every $k\ge 1$.
As explained at the beginning of \cite[Section 3]{PZ}, $\mathscr{S}_0^\lambda$ is dense in $L^2(\R^d;dx)$, and it has been frequently used
in the study of
the stochastic homogenization problem.
For any $f\in \mathscr{S}_0^\lambda$, let
 $u_k^{\w}\in L^2(k^{-1}\Z^d;\mu^{(k)})$
be the unique weak
solution of the following scaled resolvent equation
\begin{equation}\label{e3-2}
(\lambda -\sL^{(k),\w }) u^{\w}_k(x)=f(x),
\quad  x\in k^{-1}\Z^d.
\end{equation}
That is, $u^{\w}_k$ is the $\lambda$-resolvent of $\sL^{(k),\w }$ for the function $f$ in $L^2(k^{-1}\Z^d;\mu^{(k)})$, which we denote
by $R_\lambda^{(k), \w}f$
for simplicity.
With a little abuse of notation, we
extend $u^{\w}_k$ to $u^{\w}_k: \R^d \to \R$ by setting $u^{\w}_k(x)=u^{\w}_k(z)$ if $x\in \Pi_{1\le i\le d}(z_i,z_i+k^{-1}]$ for
the unique $z:=(z_1,z_2,\cdots, z_d)\in k^{-1}\Z^d$.

The main result of this paper is stated as follows.

\begin{theorem}\label{T:main}
Assume that {\bf Assumption (H1)} holds and that $d>3$.
For any $\lambda >0$,  $f\in \mathscr{S}_0^\lambda$ and $\gamma>0$,
there are a constant $C_0>0$
{\rm (}which depends on $f$, $\gamma$ and $\lambda${\rm )}
and a random variable
$k_0(\w)\ge1$ {\rm (}which depends on $\gamma$ and $\lambda$ but is independent of $f${\rm )}
such that for all $k\ge k_0(\w)$,
\begin{equation}\label{t3-1-1-1}
\|
R_\lambda^{(k), \w}f - \bar R_\lambda f
\|_{L^2(\R^d;dx)} \le C_0
(\log k)^{-\frac{1}{2}+\frac{1+\gamma}{2(d-2)}}.
\end{equation}
\end{theorem}

We will give some comments on Theorem \ref{T:main}.

\begin{itemize}
\item [(i)] Similar to the setting of $\alpha$-stable-like random walk, whose infinitesimal generator is
of the form
\begin{equation*}
\sL^\w f(x)=
\sum_{y\in \Z^d: y\not=x}
\left(f(y)-f(x)\right)\frac{w_{x,y}(\w)}{|x-y|^{d+\alpha}}
\end{equation*}
with $\alpha\in (0,2)$, the global corrector
(which has been widely used in the study of stochastic homogenization for nearest-neighbor random walks, see e.g. \cite{AKM,GNO}) may not exist, since
the jumping kernel is not $L^2$-integrable. Instead we will construct a localized corrector, and the blow up rate
for these correctors (when the corresponding domains tend to the whole space $\Z^d$) is crucial for the quantitative stochastic homogenization, see \eqref{e2-1} and Proposition \ref{p2-1}.

\item [(ii)] As seen from Theorem \ref{T:main}, in the critical jump index case the convergence rate is of order $(\log k)^{-\frac{1}{2}+\frac{1+\gamma}{2(d-2)}}$, which is slower than that for $\alpha$-stable-like random walks obtained in
\cite[Theorem 1.1]{CCKW2025} (that behaves like the polynomial decay up to logarithmic
corrections). One of main differences here is that the limit operator $\bar \sL$
is no longer a non-local operator, but a second order differential operator with constant coefficients. Due to this point,
we essentially employ a third-order Taylor expansion combined with a suitable truncation argument, which
induces a slower convergence rate. On the other hand, from these arguments, it can be noticed that the limit differential operator $\bar \sL$
does not depend on the localized corrector, which is totally different from the case of nearest-neighbor random walks
investigated by \cite{AKM,GNO}. We shall emphasize that the effective constant in
 the limit operator $\sL$ defined by \eqref{e:1.2} is just the mean of the conductance, and
we do not require that $\{w_{x,y}: (x,y)\in E\}$ necessarily follow the identical distribution. See Section~\ref{section3} for details.

\item [(iii)] Compared with \cite{CCKW2025}, the other main difficulty here is that we can only obtain the
local weak-type Poincar\'e inequality \eqref{l2-1-1}, which means that the radius of integral domain on the right hand side
of the local Poincar\'e inequality \eqref{l2-1-1} is larger than that on the left hand side. For this reason, we have to employ more technical procedures
to compare the difference between the functions $u_{k}$ and $v_{k}$ in the proof of Theorem \ref{T:main}.
Meanwhile, we also want to emphasize that in order to establish the multi-scale Poincar\'e inequality \eqref{l2-2-1},
the local strong-type Poncar\'e inequality \eqref{l2-1-1a} is enough for our purpose, where the radiuses of integral domain
are the same on both side of \eqref{l2-1-1a}, although the constant $R^2$ involved in is not optimal.

\end{itemize}

\ \

The rest of the paper is arranged as follows. In Section \ref{section2}, we establish a local weak-type Poincar\'e inequality and a multi-scale Poincar\'e inequality
for the Dirichlet form $(\E^{\w}, \F^{\w})$. In Section \ref{section2-}, we construct the localized corrector and obtain its energy estimate.
In Section \ref{section3}, we present quantitative estimates for the difference among the scaled operators $\{\sL^{(k),\w}\}_{k\ge1}$  and the limit operator $\bar \sL$.
The last section is devoted to the proof of Theorem \ref{T:main}.

\section{Two Poincar\'e-type inequalities}\label{section2}Throughout this paper,
we use := as a way of definition.
For any $k\in \N$, let $\mu^{(k)}$ be the normalized counting measure on
$k^{-1}\Z^d$  so that $\mu^{(k)} ((0, 1]^d)=1$.
For $f: k^{-1}\Z^d \to \R$, set \begin{equation}\label{e:cou-k}\mu^{(k)}(f):=\int_{k^{-1}\Z^d} f(x)\,\mu^{(k)}(dx)=k^{-d}\sum_{x\in k^{-1}\Z^d} f(x).\end{equation}
For any subset $U\subset \Z^d$ and any
$\R^d$-valued
function $f: U \to \R^n$ with some $n\ge 1$, define
\begin{equation}\label{e1-1a}
\begin{split}
	\oint_U f\,d\mu=\oint_U f(x)\mu(dx)
	:=\frac{1}{\mu(U)}\sum_{x\in U} f(x),     \\
	\sL_U^{\w} f(x):=\sum_{y\in U: y\neq x}\left(f(y)-f(x)\right)\frac{w_{x,y}(\w)}{|x-y|^{d+\alpha}}
	,\,\,x\in U,
\end{split}
\end{equation}
and
$$
\mathscr{E}_U^{\w}(f,f):=\sum_{i=1}^n
\E^\w_{U}(f^{(i)},f^{(i)}):=\frac{1}{2}\sum_{i=1}^n\sum_{x,y\in U: x\neq y}(f^{(i)}(x)-f^{(i)}(y))^2\frac{w_{x,y}(\w)}{|x-y|^{d+2}}.
$$ It is not difficult to verify that for every finite subset $U\subset \Z^d$ and $f:U \to \R^n$,
\begin{equation}\label{e1-3a}
\int_{U}\left\langle -\LL_U^\w f(x), f(x)\right\rangle\mu(dx)= \mathscr{E}_U^{\w}(f,f),
\end{equation}
where $\left\langle\cdot, \cdot \right\rangle$ is the inner product on $\R^d$.
For simplicity of notation, sometime we will omit the parameter
$\w$ when there is no danger of confusion.

For
each $R>0$ and $x\in \R^d$, set $B_R(x):=x+(-R,R]^d$, and
with a little abuse of notation but it should be clear from the context,
we also use $B_R(x)$ to denote
$B_R(x) \cap \Z^d$ or $B_R(x) \cap  (k^{-1}\Z^d)$
for $k\ge 1$. For simplicity, we write $B_R=B_R(0)$.

In this section,
we
establish two Poincar\'e-type inequalities; that is,
a local weak-type Poincar\'e inequality and a multi-scale Poincar\'e inequality
for the Dirichlet form $(\E^{\w}, \F^{\w})$ under {\bf Assumption (H1)}.

\begin{proposition}\label{l2-1}{\bf(Local weak-type Poincar\'e inequality)} Under {\bf Assumption (H1)}, there exist constants
$C_1>0$, $\kappa_0>1$ and $R_0\ge 2$ so that
for all $y\in \Z^d$, $f:\Z^d \to \R$ and $R\ge R_0$,
\begin{equation}\label{l2-1-1}
\int_{B_R(y)}\Big(f(x)-
\oint_{B_R(y)}f\,d\mu
\Big)^2\,\mu(dx)\le\frac{ C_1R^2}{\log R}\mathscr{E}_{B_{\kappa_0R}(y)}^{\w}(f,f);
\end{equation}
moreover, there is a constant $C_2>0$ such that for all $y\in \Z^d$, $f:\Z^d\to \R$ and $R\ge 1$,
\begin{equation}\label{l2-1-1a}
\int_{B_R(y)}\Big(f(x)-
\oint_{B_R(y)}f\,d\mu
\Big)^2\,\mu(dx)\le C_2R^2\mathscr{E}_{B_{R}(y)}^{\w}(f,f).
\end{equation}
\end{proposition}
\begin{proof}
Without loss of generality, it suffices to prove \eqref{l2-1-1} and \eqref{l2-1-1a} for the case $y=0$.

(1) Recall that $B_R:=(-R,R]^d$. For every $R\ge 1$, let $Q_R:=\{z\in \R^d:|z|\le R\}$,  where $|z|:=\sqrt{\sum_{i=1}^d |z_i|^2}$ is the Eulidean norm of $\R^d$. It is clear that
\begin{equation}\label{l2-1-2}
B_{\frac{R}{\sqrt{d}}}\subset  Q_R \subset \bar B_R,\quad R\ge 2\sqrt{d}.
\end{equation}

According to \cite[Proposition 3.2]{BKKL}, for any $\alpha\in (0,2)$, $R\ge 2\sqrt{d}$ and any bounded Lipschitz continuous function
$g$ on $\R^d$,
\begin{align*}
&\int_{Q_R}\int_{Q_R}\left(g(x)-g(y)\right)^2\,dx\,dy\\
&\le \frac{c_1R^{d+2}}{\log R}\int_{Q_{3R}}\int_{Q_{3R}}
\left(g(x)-g(y)\right)^2\left(\frac{1}{|x-y|^{d+\alpha}}\I_{\{|x-y|\le 1\}}+\frac{1}{|x-y|^{d+2}}\I_{\{|x-y|> 1\}}\right)\,dx\,dy.
\end{align*}
Combining this with \eqref{l2-1-2} immediately  yields that
\begin{equation}\label{l2-1-3}
\begin{split}
&\int_{B_{\frac{R}{\sqrt{d}}}}\int_{B_{\frac{R}{\sqrt{d}}}}\left(g(x)-g(y)\right)^2\,dx\,dy\\
&\le \frac{c_1R^{d+2}}{\log R}\int_{B_{3R}}\int_{B_{3R}}
\left(g(x)-g(y)\right)^2\left(\frac{1}{|x-y|^{d+\alpha}}\I_{\{|x-y|\le 1\}}+\frac{1}{|x-y|^{d+2}}\I_{\{|x-y|> 1\}}\right)\,dx\,dy.
\end{split}
\end{equation}

Given a bounded measurable function $f:\Z^d \to \R$, one can find a multi-linear extension $g:\R^d \to \R$ such that
\begin{itemize}
\item [(i)] $g$ is Lipschitz continuous and
\begin{equation}\label{l2-1-4}
g(z)=f(x),\quad z\in B_{1/4}(x) \hbox{ and } x\in \Z^d.
\end{equation}
\item [(ii)] For any $x\in \Z^d$ and $z\in B_{1/2}(x)$, the value $g(z)$
is a convex combination of the values from $\{f(y^l)\}_{1\le l \le 3^d}$, where $y^l\in \Z^d$ satisfies $\|y^l-x\|_\infty:=\sup_{1\le i \le d}|y^l_i-x_i|\le 1$.
\item [(iii)] For every $x\in \Z^d$ and $z, z'\in B_1(x)$,
\begin{equation}\label{l2-1-4a}
|g(z)-g(z')|\le \max\left\{|f(y^{l_1})-f(y^{l_2})|: y^{l_1},y^{l_2}\in \Z^d \hbox{ with }\|y^{l_1}-x\|_\infty\le 1 \hbox{ and } \|y^{l_2}-x\|_\infty\le 1\right\} |z-z'|.
\end{equation}\end{itemize}
Hence, by \eqref{l2-1-4},
\begin{align*}
&\int_{B_{\frac{R}{\sqrt{d}}}}\int_{B_{\frac{R}{\sqrt{d}}}}\left(g(x)-g(y)\right)^2\,dx\,dy\\
& \ge \sum_{x,y\in \Z^d: B_{1/2}(x)\subset B_{\frac{R}{\sqrt{d}}},
B_{1/2}(y)\subset B_{\frac{R}{\sqrt{d}}}}\int_{B_{1/4}(x)}\int_{B_{1/4}(y)}\left(g(z)-g(z')\right)^2\,dz\,dz'\\
&=|B_{1/4}|^2\sum_{x,y\in \Z^d: B_{1/2}(x)\subset B_{\frac{R}{\sqrt{d}}},
B_{1/2}(y)\subset B_{\frac{R}{\sqrt{d}}}}\left(f(x)-f(y)\right)^2\\
&\ge c_2\sum_{x,y\in B_{\frac{R}{\sqrt{d}}-1}}\left(f(x)-f(y)\right)^2.
\end{align*}
On the other hand, since for every $x\in \Z^d$ and $z\in B_{1/2}(x)$, $g(z)$ is a convex combination of the values $\{f(y^l)\}_{1\le l \le 3^d}$
with $\|y^l-x\|_\infty\le 1$ for all $1\le l \le 3^d$, we obtain that for any  $x,x'\in \Z^d$, $z\in B_{1/2}(x)$ and $z'\in B_{1/2}(x')$,
\begin{align*}
|g(z)-g(z')|\le
\max\left\{|f(y^{l_1})-f(y^{l_2})|: y^{l_1},y^{l_2}\in \Z^d \hbox{ with } \|y^{l_1}-x\|_\infty\le 1 \hbox{ and } \|y^{l_2}-x'\|_\infty\le 1\right\}.
\end{align*}
Thus, combining this with \eqref{l2-1-4a} yields that
\begin{align*}
&\int_{B_{3R}}\int_{B_{3R}}
\left(g(x)-g(y)\right)^2\left(\frac{1}{|x-y|^{d+\alpha}}\I_{\{|x-y|\le 1\}}+\frac{1}{|x-y|^{d+2}}\I_{\{|x-y|> 1\}}\right)dxdy\\
&\le c_3\sum_{x,y\in B_{3R+1}}\frac{\left(f(x)-f(y)\right)^2}{|x-y|^{d+2}}.
\end{align*}

Putting all the estimates above together, taking such $g$ into \eqref{l2-1-3} and re-arranging the value of
$R$, we can obtain that there are $R_0\ge 2$ and $\kappa_0>3$ such that for every $R\ge R_0$ and $f:\Z^d\to \R$,
\begin{equation*}
\sum_{x,y\in B_R}\left(f(x)-f(y)\right)^2
\le \frac{c_4R^{d+2}}{\log R}\sum_{x,y\in B_{\kappa_0R}}
\frac{\left(f(x)-f(y)\right)^2}{|x-y|^{d+2}}.
\end{equation*}
Therefore, for every $R\ge R_0$,
\begin{align*}
&\int_{B_R}\Big(f(x)-
\oint_{B_R}f\,d\mu
\Big)^2\,\mu(dx)\\
&=
\mu(B_R)^{-2}\sum_{x\in B_R}\bigg(\sum_{x'\in B_R}
(f(x)-f(x'))\bigg)^2\le
 \mu(B_R)^{-1}\sum_{x,x'\in B_R}\left(f(x)-f(x')\right)^2\\
&\le  \frac{c_5R^{2}}{\log R}\sum_{x,x'\in B_{\kappa_0R}}
\frac{\left(f(x)-f(x')\right)^2}{|x-x'|^{d+2}}\le \frac{c_6R^{2}}{\log R}\E^\w_{B_{\kappa_0 R}}(f,f).
\end{align*}
where the last inequality is due to \eqref{a1-1-1}. So, we prove \eqref{l2-1-1}.

(2) For all $R\ge1$ and $f:\Z^d \to \R$,
\begin{align*}
& \int_{B_R}\Big(f(x)-
\oint_{B_R}f\,d\mu
\Big)^2\,\mu(dx)\\
&=
\mu(B_R)^{-2}\sum_{x\in B_R}\bigg(\sum_{x'\in B_R}
(f(x)-f(x'))\bigg)^2\\
&\le \mu(B_R)^{-2}\sum_{x\in B_R}\left(\sum_{x'\in B_R}\left(f(x)-f(x')\right)^2\frac{w_{x,x'}}{|x-x'|^{d+2}}\right)
\cdot\left(\sum_{x'\in B_R}w_{x,x'}^{-1}|x-x'|^{d+2}\right)\\
&\le  c_7R^2\sum_{x,x'\in B_{R}}
\left(f(x)-f(x')\right)^2\frac{w_{x,x'}}{|x-x'|^{d+2}}\le c_8R^2\E^\w_{B_{R}}(f,f),
\end{align*}
which proves \eqref{l2-1-1a}.
\end{proof}

Next, we will establish
a
multi-scale Poincar\'e inequality, whose proof is partly motivated by that of \cite[Proposition 1.7]{AKM}. For every
$m,n\in \N$ with $n<m$, define
$$
\Z_{m,n}^d:=\left\{z=(z_1,\cdots,z_d)\in B_{2^m}: z_i=k_i2^n\hbox{ for some odd}\ k_i\in
\Z
\hbox{ for all } 1\le i \le d\right\}.
$$
When $m=n$, we use the convention  $\Z_{m,m}^d=\{0\}$.

\begin{proposition}\label{l2-2}{\bf(Multi-scale Poincar\'e inequality)}
Under {\bf Assumption (H1)}, there is a constant $C_2>0$ such that
for every $m\ge 1$, $1\le n\le m-1$,
and $f, g: B_{2^m}\to \R$,
\begin{equation}\label{l2-2-1}
\begin{split}
\sum_{x\in B_{2^m}}f(x)\bigg(g(x)-
\oint_{B_{2^m}}g \,d\mu
\bigg)\le&
\sum_{z\in \Z_{m,n}^d}\sum_{x\in B_{2^n}(z)}f(x)\bigg(g(x)-
\oint_{B_{2^n}(z)}g\,d\mu
\bigg)\\
&+C_2\mathscr{E}_{B_{2^m}}^\w(g,g)^{1/2}\sum_{k=n}^{m-1} 2^{k(d+2)/2}\bigg(\sum_{y\in \Z_{m,k}^d}\bigg(\oint_{B_{2^k}(y)}f\,d\mu\bigg)^2\bigg)^{1/2}.
\end{split}
\end{equation}
\end{proposition}
\begin{proof} For any $1\le k<m$ and $f,g:B_{2^m}\to \R$,
\begin{equation}\label{l2-2-1a}
\begin{split}
&\sum_{z\in \Z_{m,k+1}^d}\int_{B_{2^{k+1}}(z)}f(x)\bigg(g(x)-\oint_{B_{2^{k+1}}(z)} g\,d\mu\bigg)\,\mu(dx)\\
&=\sum_{z\in \Z_{m,k+1}^d}\sum_{y\in \Z_{m,k}^d\cap B_{2^{k+1}}(z) }\int_{B_{2^{k}}(y)}f(x)\bigg(g(x)-\oint_{B_{2^{k+1}}(z)} g\,d\mu\bigg)\,\mu(dx)\\
&=\sum_{z\in \Z_{m,k+1}^d}\sum_{y\in \Z_{m,k}^d\cap B_{2^{k+1}}(z) }
\Bigg[\int_{B_{2^{k}}(y)}f(x)\left(g(x)-\oint_{B_{2^{k}}(y)} g\,d\mu\right)\,\mu(dx)\\
&\qquad \qquad\qquad\qquad\qquad\qquad+
\int_{B_{2^{k}}(y)}f(x)\left(\oint_{B_{2^{k}}(y)} g\,d\mu-\oint_{B_{2^{k+1}}(z)} g\,d\mu\right)\,\mu(dx)\Bigg]\\
&=\sum_{y\in \Z_{m,k}^d}\int_{B_{2^{k}}(y)}f(x)\left(g(x)-\oint_{B_{2^{k}}(y)} g\,d\mu\right)\,\mu(dx)\\
&\quad +\mu(B_{2^k})
\sum_{z\in \Z_{m,k+1}^d}\sum_{y\in \Z_{m,k}^d\cap B_{2^{k+1}}(z) }
\left(\oint_{B_{2^k}(y)}f\,d\mu\right)\left(\oint_{B_{2^{k}}(y)} g\,d\mu-\oint_{B_{2^{k+1}}(z)} g\,d\mu\right).
\end{split}
\end{equation}

According to the Cauchy-Schwartz inequality,
\begin{equation}\label{l2-2-1ab}
\begin{split}
&\mu(B_{2^k})\sum_{z\in \Z_{m,k+1}^d}\sum_{y\in \Z_{m,k}^d\cap B_{2^{k+1}}(z) }
\left(\oint_{B_{2^k}(y)}f\,d\mu\right)\left(\oint_{B_{2^{k}}(y)} g\,d\mu-\oint_{B_{2^{k+1}}(z)} g\,d\mu\right)\\
&\le c_12^{kd}\left(\sum_{z\in \Z_{m,k+1}^d}\sum_{y\in \Z_{m,k}^d\cap B_{2^{k+1}}(z) }
\left(\oint_{B_{2^{k}}(y)} g\,d\mu-\oint_{B_{2^{k+1}}(z)} g\,d\mu\right)^2\right)^{1/2}
\left(\sum_{y\in \Z_{m,k}^d}\left(\oint_{B_{2^k}(y)}f\,d\mu\right)^2\right)^{1/2} \\
 &\le c_12^{kd}\left(\sum_{z\in \Z_{m,k+1}^d}\sum_{y\in \Z_{m,k}^d\cap B_{2^{k+1}}(z) }
  \oint_{B_{2^{k}}(y)} \left(g(x)-\oint_{B_{2^{k+1}}(z)} g\,d\mu\right)^2 \,\mu(dx)    \right)^{1/2}\\
  &\quad\times
\left(\sum_{y\in \Z_{m,k}^d}\left(\oint_{B_{2^k}(y)}f\,d\mu\right)^2\right)^{1/2} .
\end{split}
\end{equation}
Furthermore,
\begin{align*}
&  \oint_{B_{2^{k}}(y)} \left(g(x)-\oint_{B_{2^{k+1}}(z)} g\,d\mu\right)^2 \,\mu(dx)\\
&=\mu(B_{2^k}(y))^{-1}\mu(B_{2^{k+1}}(z))^{-2}\sum_{x\in B_{2^k}(y)}
\left(\sum_{x'\in B_{2^{k+1}}(z)}(g(x)-g(x'))\right)^2\\
&\le c_22^{-3kd}\sum_{x\in B_{2^k}(y)}
\left(\sum_{x'\in B_{2^{k+1}}(z)}(g(x)-g(x'))^2\frac{w_{x,x'}}{|x-x'|^{d+2}}\right)\cdot\left(\sum_{x'\in B_{2^{k+1}}(z)}w_{x,x'}^{-1}|x-x'|^{d+2}\right)\\
&\le c_32^{-k(d-2)}\int_{B_{2^k}(y)}\int_{B_{2^{k+1}}(z)}(g(x)-g(x'))^2\frac{w_{x,x'}}{|x-x'|^{d+2}}\,\mu(dx)\,\mu(dx'),
\end{align*}
where the last inequality follows from \eqref{a1-1-1}.
Thus,
\begin{align*}
 & \sum_{z\in \Z_{m,k+1}^d}\sum_{y\in \Z_{m,k}^d\cap B_{2^{k+1}}(z) }
\oint_{B_{2^{k}}(y)} \left(g(x)-\oint_{B_{2^{k+1}}(z)} g\,d\mu\right)^2 \,\mu(dx) \\
 &\le c_42^{-k(d-2)}\sum_{z\in \Z_{m,k+1}^d}\sum_{y\in \Z_{m,k}^d\cap B_{2^{k+1}}(z) }
\int_{B_{2^k}(y)}\int_{B_{2^{k+1}}(z)}(g(x)-g(x'))^2\frac{w_{x,x'}}{|x-x'|^{d+2}}\,\mu(dx)\,\mu(dx')\\
&\le c_52^{-k(d-2)}
\int_{B_{2^m}}\int_{B_{2^m}}(g(x)-g(x'))^2\frac{w_{x,x'}}{|x-x'|^{d+2}}\,\mu(dx)\,\mu(dx')=2c_52^{-k(d-2)}
\mathscr{E}_{B_{2^m}}^\w(g,g).
\end{align*}

Combining this with \eqref{l2-2-1ab} and \eqref{l2-2-1a}, and taking the summation in \eqref{l2-2-1a} from $k=n$
with $1\le n\le m-1$
to $k=m-1$,
we get the desired inequality \eqref{l2-2-1}.
\end{proof}

\section{Local correctors and energy estimates} \label{section2-}

As mentioned above, due to the lack of the $L^2$-integrability of jumping kernel, it seems that the
global corrector may not exist. Instead, we will construct a local corrector, whose averaged $L^2$-norm
will blow up as the associated region tends to the whole space $\Z^d$. The speed of such blow up behavior will make
a crucial role in the quantitative homogenization.

Suppose that {\bf Assumption (H1)} holds.
Note that $B_{2^m}$ is a finite set in $\Z^d$, and so there exists a (unique) solution $\phi_m:B_{2^m}\to \R^d$ to the following equation
\begin{equation}\label{e2-1}
\begin{cases}
 \sL^\w_{B_{2^m}} \phi_m (x)=-V(x)+\displaystyle
 \oint_{B_{2^m}} V\,d\mu,
 \quad  x\in B_{2^m},\\
 \sum_{x\in B_{2^m}}  \phi_m(x)=0,
\end{cases}
\end{equation}
where $\sL^\w_{B_{2^m}}$ is defined by \eqref{e1-1a} with $U=B_{2^m}$, and
\begin{equation}\label{e2-2}
V(x):=\sum_{z\in \Z^d} \frac{z}{|z|^{d+2}}w_{x,x+z}.
\end{equation}
The solution $\phi_m$ is called the local corrector as $V$ can be seen as the outcome of $\sL^\w$ over the affine functions.

\begin{proposition}\label{p2-1} Assume that {\bf Assumption (H1)} holds.
Assume that $d\ge 3$.
Then, for any $\gamma>0$, there exist a constant $C_1>0$ and a random variable $m_0(\omega)\ge1$ such that for every $m\ge m_0(\w)$,
\begin{equation}\label{p2-1-2}
 \mathscr{E}_{B_{2^m}}^\w(\phi_m,\phi_m) \le C_1
m^{\frac{(1+\gamma)}{d-2}}2^{md}.
\end{equation}
\end{proposition}
\begin{proof}
We write $\phi_m(x)$ and $V(x)$ as $\phi_m(x):=(\phi_m^{(1)}(x),\phi_m^{(2)}(x),\cdots,\phi_m^{(d)}(x))$ and
$V(x)=(V^{(1)}(x),V^{(2)}(x),\cdots,$ $V^{(d)}(x))$, respectively. In particular,
$$\mathscr{E}_{B_{2^m}}^\w(\phi_m,\phi_m)=\sum_{i=1}^d \mathscr{E}_{B_{2^m}}^\w(\phi_m^{(i)},\phi_m^{(i)}).$$
According to
\eqref{e1-3a} and
\eqref{e2-1}, for $1\le i\le d$,
\begin{align*}
\mathscr{E}_{B_{2^m}}^\w(\phi_m^{(i)},\phi_m^{(i)})&=-
\int_{B_{2^m}}\sL^\w_{B_{2^m}}\phi_m^{(i)}(x)\cdot
\phi_m^{(i)}(x)\,\mu(dx)= \int_{B_{2^m}}\left(V^{(i)}(x)-\oint_{B_{2^m}}V^{(i)}\,d\mu\right)\cdot\phi_m^{(i)}(x)\,\mu(dx).
\end{align*}

Below, we take
$f=V^{(i)}-\displaystyle\oint_{B_{2^m}}V^{(i)}\,d\mu$ and $g=\phi_m^{(i)}$ in
\eqref{l2-2-1}, and  obtain that for every $m\ge 1$ and $1\le n\le m-1$,
\begin{align*}
 \mathscr{E}_{B_{2^m}}^\w(\phi_m,\phi_m)
\le & \sum_{i=1}^d\sum_{y\in \Z_{m,n}^d}\int_{B_{2^n}(y)}
\left(V^{(i)}(x)-\oint_{B_{2^m}}V^{(i)}\,d\mu\right)\cdot\left(\phi_m^{(i)}(x)-\oint_{B_{2^n}(y)}\phi_m^{(i)}\,d\mu\right)\,\mu(dx)\\
&+c_1\sum_{i=1}^d\mathscr{E}_{B_{2^m}}^\w(\phi_m^{(i)},\phi_m^{(i)})^{1/2}\sum_{k=n}^{m-1} 2^{k(d+2)/2}\left(\sum_{y\in \Z_{m,k}^d}\left(\oint_{B_{2^k}(y)}V^{(i)}\,d\mu
-\oint_{B_{2^m}}V^{(i)}\,d\mu\right)^2\right)^{1/2}\\
= &:J_{m,n,1}+J_{m,n,2}.
\end{align*}

Let $C_1>0$ be the constant in \eqref{l2-1-1a}. Then, applying Young's inequality, we derive
\begin{align*}
J_{m,n,1}&\le  2C_12^{2n}\sum_{y\in \Z_{m,n}^d}\int_{B_{2^n}(y)}\left|V(x)-\oint_{B_{2^m}}V\,d\mu\right|^2\mu(dx) \\
&\quad+\frac{2^{-2n }}{8C_1}\sum_{z\in \Z_{m,n}^d}\int_{B_{2^n}(y)}\left|\phi_m(x)-\oint_{B_{2^n}(y)}\phi_m\,d\mu\right|^2\mu(dx)\\
&\le c_22^{md+2n }+\frac{1}{8}
\sum_{i=1}^d
\sum_{z\in \Z_{m,n}^d}\int_{B_{2^n}(z)}
\int_{B_{2^n}(z)}(\phi_m^{(i)}(x)-\phi_m^{(i)}(y))^2\frac{w_{x,y}}{|x-y|^{d+2}}\,\mu(dx)\,\mu(dy)\\
&\le c_22^{md+2n}+
\frac{1}{4} \mathscr{E}_{B_{2^m}}^\w(\phi_m,\phi_m),
\end{align*}
where in the second inequality we used \eqref{l2-1-1a} and the fact that $\sup_{x\in \Z^d}|V(x)|\le c_3$, and the last inequality is due to the fact
$\sum_{z\in\Z_{m,n}^d} \mathscr{E}_{B_{2^n}}^\w(\phi_m^{(i)},\phi_m^{(i)})
\le \mathscr{E}_{B_{2^m}}^\w(\phi_m^{(i)},\phi_m^{(i)})$ since
$\sum_{z\in \Z_{m,n}^d} B_{2^n}(z) =B_{2^m}$.

On the other hand, for every $1\le k \le m$,
\begin{align*}
\sum_{y\in \Z^d_{m,k}}\left|\oint_{B_{2^k}(y)}V\,d\mu\right|^2&=2^{-2(k+1)d}\sum_{i=1}^d\sum_{y\in \Z^d_{m,k}}\sum_{z_1,z_2\in B_{2^k}(y)}V^{(i)}(z_1)V^{(i)}(z_2)\\
&=2^{-2(k+1)d}\sum_{i=1}^d\sum_{y\in \Z^d_{m,k}}\sum_{z_1,z_2\in B_{2^k}(y)}
\sum_{z_3,z_4\in \Z^d}\frac{(z_1-z_3)^{(i)}}{|z_1-z_3|^{d+2}}
\frac{(z_2-z_4)^{(i)}}{|z_2-z_4|^{d+2}}\xi_{z_1,z_3}\xi_{z_2,z_4}.
\end{align*}
Here, $\xi_{x,y}:=w_{x,y}-1$ for all $x,y\in \Z^d$, and  in the second equality follows from the fact that
$$V(x)=\sum_{z\in \Z^d}\frac{z}{|z|^{d+2}}w_{x,x+z}=\sum_{z\in \Z^d}\frac{z}{|z|^{d+2}}
\left(w_{x,x+z}-1\right).$$
Then,
\begin{align*}
&\Ee\left[\left(\sum_{y\in \Z^d_{m,k}}\left|\oint_{B_{2^k}(y)}V\,d\mu\right|^2\right)^2\right]\\
&=2^{-4(k+1)d}\sum_{i,i'=1}^d\sum_{y,y'\in \Z^d_{m,k}}\sum_{z_1,z_2\in B_{2^k}(y)}\sum_{z_1',z_2'\in B_{2^k}(y')}
\sum_{z_3,z_4,z_3',z_4'\in \Z^d}
\frac{(z_1-z_3)^{(i)}(z_1'-z_3')^{(i')}}{(|z_1-z_3||z_1'-z_3'|)^{d+2}}\\
&\hskip 3truein \times
\frac{(z_2-z_4)^{(i)}(z_2'-z_4')^{(i')}}{(|z_2-z_4||z_2'-z_4'|)^{d+2}}
\Ee[\xi_{z_1,z_3}\xi_{z_2,z_4}\xi_{z_1',z_3'}
\xi_{z_2',z_4'}
].
\end{align*}
Note that $\Ee[\xi_{x,y}]=0$. According to the independence of $\{w_{x,y}\}_{(x,y)\in E}$, it is easy to see that
\begin{equation*}
\Ee\left[\xi_{z_1,z_3}\xi_{z_2,z_4}\xi_{z_1',z_3'}\xi_{z_2',z_4'}\right]\neq 0
\end{equation*}
only if
one of the following three cases happens:
\begin{itemize}
\item [(a)] $(z_1,z_3)=(z_2,z_4)$, $(z_1',z_3')=(z_2',z_4')$;
\medskip
\item [(b)] $(z_1,z_3)=(z_1',z_3')$, $(z_2,z_4)=(z_2',z_4')$;
\medskip
\item [(c)] $(z_1,z_3)=(z_2',z_4')$, $(z_2,z_4)=(z_1',z_3')$.
\end{itemize}
Hence,
\begin{align*}
& \Ee\left[\left(\sum_{y\in \Z^d_{m,k}}\left|\oint_{B_{2^k}(y)}V\,d\mu\right|^2\right)^2\right]\le
c_4\sum_{l=1}^2 I_{l,k,m},
\end{align*}
where
\begin{align*}
I_{1,k,m}:=& 2^{-4(k+1)d}\sum_{i,i'=1}^d\sum_{y,y'\in \Z^d_{m,k}}\sum_{(z_1,z_3)\in E: z_1\in B_{2^k}(y)}
\sum_{(z_1',z_3')\in E: z_1'\in B_{2^k}(y')}\frac{|(z_1-z_3)^{(i)}|^2}{|z_1-z_3|^{2(d+2)}}\\
& \hskip 3truein
\times  \frac{|(z_1'-z_3')^{(i')}|^2}{|z_1'-z_3'|^{2(d+2)}}
\Ee[|\xi_{z_1,z_3}|^2 |\xi_{z_1',z_3'}|^2]
\end{align*} and
\begin{align*}
I_{2,k,m}:=& 2^{-4(k+1)d}\sum_{i,i'=1}^d\sum_{y\in \Z^d_{m,k}}\sum_{(z_1,z_3)\in E: z_1\in B_{2^k}(y)}
\sum_{(z_2,z_4)\in E: z_2\in B_{2^k}(y)}\\
&\qquad\quad \times \frac{|(z_1-z_3)^{(i)}||(z_1-z_3)^{(i')}|}{ |z_1-z_3| ^{2(d+2)}} \frac{|(z_2-z_4)^{(i)}||(z_2-z_4)^{(i')}|}{ |z_2-z_4| ^{2(d+2)}}
\Ee[|\xi_{z_1,z_3}|^2|\xi_{z_2,z_4}|^2].
\end{align*}
Since $|\xi_{x,y}|\le c_5$ and $\mu(\Z_{m,k}^d)=2^{d(m-k)}$, we can deduce that
for every $1\le k \le m$ and $y\in \Z_{m,k}^d$,
\begin{align*}
I_{1,k,m}&\le c_62^{-4(k+1)d}
\sum_{y,y'\in \Z^d_{m,k}}\sum_{(z_1,z_3)\in E: z_1\in B_{2^k}(y)}
\sum_{(z_1',z_3')\in E: z_1'\in B_{2^k}(y')}\frac{1}{(|z_1-z_3||z_1'-z_3'|)^{2(d+1)}}\\
&\le c_72^{-2kd+2(m-k)d}
\end{align*} and
\begin{align*}
I_{2,k,m}
&\le c_62^{-4(k+1)d}\sum_{y\in \Z^d_{m,k}}
\sum_{(z_1,z_3)\in E: z_1\in B_{2^k}(y)}
\sum_{(z_2,z_4)\in E: z_2\in B_{2^k}(y)}\frac{1}{(|z_1-z_3||z_2-z_4|)^{2(d+1)}}\\
&\le c_72^{-2kd+(m-k)d}.
\end{align*}
Thus,
\begin{align*}
&\Ee\left[\left(\sum_{y\in \Z_{m,k}^d}\left|\oint_{B_{2^k}(y)}V\,d\mu\right|^2\right)^2\right]\le
c_82^{-2kd+2(m-k)d}.
\end{align*}
By the Markov inequality, we can see that for every $\theta\in ({2}/{d},1)$ and $\gamma>0$,
\begin{align*}
&\Pp\left(\bigcup_{m\ge 1}\bigcup_{[\frac{(1+\gamma)\log_2 m}{2(1-\theta)d}]\le k \le m}\left\{\left|
\sum_{y\in \Z_{m,k}^d}\left|\oint_{B_{2^k}(y)}V\,d\mu\right|^2\right|>2^{(m-k)d-\theta k d}\right\}\right)\\
&\le \sum_{m=1}^\infty\sum_{k=[\frac{(1+\gamma)\log_2 m}{2(1-\theta)d}]}^m2^{2\theta k d-2(m-k)d}
\Ee\left[\left(\sum_{y\in \Z_{m,k}^d}\left|\oint_{B_{2^k}(y)}V\,d\mu\right|^2\right)^2\right]\\
&\le c_8\sum_{m=1}^\infty \sum_{k=[\frac{(1+\gamma)\log_2 m}{2(1-\theta)d}]}^m 2^{-2k(1-\theta)d}\le c_9\sum_{m=1}^\infty m^{-(1+\gamma)}<\infty.
\end{align*}
Therefore, according to Borel-Cantelli's lemma, we can find $m_0(\w)\ge 1$ such that for all $m\ge m_0(\w)$ and $[\frac{(1+\gamma)\log_2 m}{2(1-\theta)d}]\le
k\le m$,
$$
\sum_{y\in \Z_{m,k}^d}\left|\oint_{B_{2^k}(y)}V\,d\mu\right|^2\le 2^{(m-k)d-{\theta kd} }.
$$ In particular,
\begin{equation}\label{p2-1-4}\left|\oint_{B_{2^m}}V\,d\mu\right|^2\le 2^{-{\theta md} }.\end{equation}
These two estimates imply that for every
$m\ge n\ge \left[\frac{(1+\gamma)\log_2 m}{2(1-\theta)d}\right]$,
\begin{align*}
J_{m,n,2}&\le c_{10}  \mathscr{E}_{B_{2^m}}^\w(\phi_m,\phi_m) ^{1/2}\left(\sum_{k=n}^{m-1} 2^{k(d+2)/2+(m-k)d/2-
kd\theta/2
} \right)\\
&=c_{10} 2^{md/2} \mathscr{E}_{B_{2^m}}^\w(\phi_m,\phi_m) ^{1/2}\left(\sum_{k=n}^{m-1}
2^{-k(\theta d-2)/2}
\right)\\
& \le c_{11}2^{md/2} \mathscr{E}_{B_{2^m}}^\w(\phi_m,\phi_m) ^{1/2}\le c_{12}2^{md}+\frac{1}{4}\mathscr{E}_{B_{2^m}}^\w(\phi_m,\phi_m),
\end{align*}
where the second inequality follows from the fact
$\theta\in (2/d,1)$.

Combining
all the estimates above, taking
$n=\left[\frac{(1+\gamma)\log_2 m}{2(1-\theta)d}\right]$ and letting $\theta$ close to $2/d$, we can  get
the desired assertion \eqref{p2-1-2} by adjusting the parameter $\gamma$.
\end{proof}

 \section{Convergence rates of scaled operators  $\{\sL^{(k),\w}\}_{k\ge1}$}\label{section3}
In this section, we will study estimates for the difference among the scaled operators  $\{\sL^{(k),\w}\}_{k\ge1}$   and the limit operator $\bar \sL$,
which is also an important ingredient for the convergence rate of the associated scaled resolvents.
Recall that $\bar \sL$ is defined by \eqref{e:1.2}. We first introduce the following operator on ${\mathcal B}_b(k^{-1}\Z^d)$ (the set of bounded measurable functions on $k^{-1}\Z^d$), which can be viewed as a discrete approximation to the limit operator $\bar \sL$:
\begin{equation}\label{e3-1a}
\begin{split}
\bar \sL^{(k)}f(x):&= k^{-d}(\log k)^{-1}\sum_{z\in k^{-1}\Z^d}\left(f(x+z)-f(x)\right)\frac{1}{|z|^{d+2}}\\
&=k^{-d}(\log k)^{-1}\sum_{z\in k^{-1}\Z^d}\left(f(x+z)-f(x)-\langle \nabla f(x), z\rangle \right)\frac{1}{|z|^{d+2}},\quad x\in k^{-1}\Z^d.
\end{split}
\end{equation}

\begin{proposition}\label{l3-1-0}
For any
$f\in C_c^2 (\R^d)$,
there is a constant $C_1>0$ such that for all $k\ge 1$,
\begin{equation}\label{l3-1-2}
\int_{k^{-1}\Z^d}  |\bar \sL^{(k)}f(x)-\bar \sL f(x) |^2\,
\mu^{(k)}(dx)\le  C_1(\log k)^{-2}.
\end{equation}
\end{proposition}

\begin{proof}
Without loss of generality, we
may
assume that
$\text{supp} [ f ] \subset B_1$.
We make the decomposition of $\bar \sL^{(k)}f$ as follows:
\begin{align*}
\bar \sL^{(k)}f(x)&=k^{-d}(\log k)^{-1}
\left(\sum_{z\in k^{-1}\Z^d:|z|\le 1}+\sum_{z\in k^{-1}\Z^d:|z|> 1}\right)
\left(\frac{f(x+z)-f(x)-\langle \nabla f(x), z\rangle}{|z|^{d+2}}\right)\\
&=:I_1^k(x)+I_2^k(x).
\end{align*}

Note that, since $\text{supp} [ f ] \subset B_1$, $I_1^k(x)=0$ when $x\notin B_3$. Applying the Taylor expansion, we obtain that
\begin{align*}
I_1^k(x)=\left\langle \nabla^2 f(x), A^k\right\rangle
+I_{11}^k(x),\quad x\in B_{3},
\end{align*}
where $A^k=\{A^k_{ij}\}_{1\le i,j\le d}$ is a $d\times d$ matrix defined by
\begin{align*}
A^k_{ij}:=\frac{k^{-d}}{2\log k}\sum_{z\in k^{-1}\Z^d:|z|\le 1}\frac{z_iz_j}{|z|^{d+2}}=
\frac{1}{2\log k}\sum_{z\in \Z^d:|z|\le k}\frac{z_iz_j}{|z|^{d+2}},\quad 1\le i,j\le d,
\end{align*}
and the remaining term $I_{11}^k(x)$ satisfies that
\begin{align*}
\sup_{x\in B_{3}}|I_{11}^k(x)|&\le \frac{\|\nabla^3 f\|_\infty}{6\log k}
\cdot\left(\sum_{z\in k^{-1}\Z^d:|z|\le 1}|z|^{-d+1}\right)\le c_1(\log k)^{-1}.
\end{align*}
Furthermore, by the symmetry property, it is not difficult to verify that
\begin{align*}
A^k_{ij}=0,\,\, i\neq j,\quad A^k_{ii}=\frac{1}{2d\log k}\left(\sum_{z\in \Z^d:|z|\le k}|z|^{-d}\right),\,\,1\le i \le d.
\end{align*}
By the mean-value theorem,
\begin{align*}
\sup_{k\ge 2}\left|\left(\sum_{z\in \Z^d:|z|\le k}|z|^{-d}\right)-\int_{\{ 1\le |z|\le k\}}|z|^{-d}\,dz\right|\le c_2
\end{align*} and
\begin{align*}
\int_{\{ 1\le |z|\le k\}}|z|^{-d}\,dz=c_0(d)\int_1^k \frac{1}{r}\,dr.
\end{align*}
According to the definition \eqref{e1-3a} for $ a_0$, we can prove that
\begin{equation*}
|A^k-a_0\I_{d\times d}|\le c_2(\log k)^{-1},
\end{equation*}
where $\I_{d\times d}$ denotes the $d\times d$ identity matrix.

On the other hand,
\begin{align*}
\sup_{x\in B_{3}}|I_2^k(x)|&\le
(\log k)^{-1}\left(2\|f\|_\infty\cdot\left(\sum_{z\in k^{-1}\Z^d}k^{-d}|z|^{-d-2}\right)+\|\nabla f\|_\infty
\cdot\left(\sum_{z\in k^{-1}\Z^d}k^{-d}|z|^{-d-1}\right)\right)\\
&\le c_3(\log k)^{-1}.
\end{align*}
Therefore, we can derive
\begin{align}\label{l3-1-1a}
\sup_{x\in B_{3}}\left|\bar \sL^{(k)} f(x) -\bar a_0\Delta f(x)\right|\le c_4(\log k)^{-1}.
\end{align}

Note that for all $x\notin B_3$, $f(x)=0$, and $f(x+z)\neq 0$ only if $z\in B_2(x)$. So, for any $x\notin B_3$,
\begin{align*}
|\bar \sL^{(k)}f(x)|&\le \|f\|_\infty(\log k)^{-1}\cdot\left(\sum_{z\in B_2(x)}|z|^{-d-2}\right)\\
&\le c_5(\log k)^{-1}(1+|x|)^{-d-2}.
\end{align*}

Combining this with \eqref{l3-1-1a} yields the desired conclusion \eqref{l3-1-2}.\end{proof}

Motivated by the definitions \eqref{e:operator} and \eqref{e3-1a}, we next introduce the following
variant
of the scaled operators  $\{\sL^{(k),\w}\}_{k\ge1}$.
For every $f\in C_b (\R^d)$,
$k\ge1$ and  $x\in k^{-1}\Z^d$,
define
\begin{equation}\label{e:operator0}
 \widehat \sL^{(k),\w}f(x):=
\displaystyle k^{-d}(\log k)^{-1}\sum_{z\in k^{-1}\Z^d}\left(f(x+z)-f(x)-
\langle \nabla f(x), z\rangle
\right)\frac{w_{kx,k(x+z)}(\w)}{|z|^{d+2}}.
\end{equation} For simplicity, we will omit
the random environment
$\w \in \Omega$ from the notation.

\begin{proposition}\label{l3-1} Under {\bf Assumption (H1)}, for every
$f\in C_c^2 (\R^d)$,
there exist a random variable $k_0(\w)\ge1$ and a constant $C_2>0$ such that for all $k\ge k_0(\w)$,
\begin{equation}\label{l3-1-1}
\int_{k^{-1}\Z^d} |\widehat \sL^{(k)}f(x)-\bar \sL^{(k)}f(x) |^2\,\mu^{(k)}(dx)\le \frac{C_2}{\log k}.
\end{equation}
\end{proposition}
\begin{proof}
Without loss of generality, we assume that $\text{supp}[ f] \subset B_1$.
For every $x,y\in \Z^d$, let
$$\xi_k(x,y):=f(k^{-1}y)-f(k^{-1}x)-k^{-1}
\langle \nabla f(k^{-1}x), (y-x)\rangle,
\quad \eta_k(x,y):=\xi_{k}(x,y)\frac{w_{x,y}-1}{|x-y|^{d+2}}.
$$
Then, for all $x\in \Z^d$,
$$
 \widehat \sL^{(k)}f(k^{-1}x)=\frac{k^2}{\log k}
 \sum_{z\in \Z^d \setminus \{x\} }
 \xi_k(x,z)\frac{w_{x,z}}{|x-z|^{d+2}},\quad  \bar \sL^{(k)}f(k^{-1}x)=\frac{k^2}{\log k}
 \sum_{z\in \Z^d \setminus \{x\} }
 \xi_k(x,z)\frac{1}{|x-z|^{d+2}}.
$$ According to the fact $\Ee[w_{x,z}]=1$ for all $x,z\in \Z^d$, we further obtain that
\begin{align}\label{l3-1-5}
\int_{k^{-1}\Z^d} |\widehat \sL^{(k)}f(x)-\bar \sL^{(k)}f(x) |^2\,\mu^{(k)}(dx)
=k^{-(d-4)}(\log k)^{-2}\sum_{x\in \Z^d}\bigg(\sum_{z\in \Z^d\setminus \{x\} }\eta_k(x,z)\bigg)^2.
\end{align}

Define $$b_k(x,y):=\frac{\left|f(k^{-1}y)-f(k^{-1}x)-k^{-1}
\langle \nabla f(k^{-1}x), (y-x)\rangle
\right|}
{|y-x|^{d+2}}=\frac{|\xi_k(x,y)|}{|y-x|^{d+2}}.$$
Then, for every $x\in B_{2k}$, by the mean-value theorem,
\begin{equation}\label{l3-1-3a-}\begin{split}
\forall x\in B_{2k}, \quad \|b_k(x)\|_2^2:&= \sum_{y\in \Z^d}|b_k(x,y)|^2\le \sum_{y\in B_{3k}}|b_k(x,y)|^2+\sum_{y\in B_{3k}^c}|b_k(x,y)|^2\\
&\le \frac{1}{4}k^{-4}\|\nabla^2 f\|_\infty^2 \sum_{y\in B_{3k}}\frac{1}{|x-y|^{2d}}\\
&\quad +
c_1\left(\|f\|_\infty^2+\|\nabla f\|_\infty^2\right)\left(
\sum_{y\in B_{3k}^c}\frac{1}{(1+|y|)^{2d+4}}+k^{-2}\sum_{z\in B_{3k}^c}\frac{1}{(1+|y|)^{2d+2}}\right)\\
&\le c_2k^{-4};
\end{split}\end{equation}
for any $x\in B_{2k}^c$, by the fact that
$\xi_k(x,y)\neq 0$ only if $y\in B_k$, thanks to ${\rm supp}[f]\subset B_1$,
\begin{equation}\label{l3-1-4a-}\begin{split}
\forall x\in B_{2k}^c, \quad \|b_k(x)\|_2^2:=\sum_{y\in \Z^d}|b_k(x,y)|^2&=\sum_{y\in B_k}|b_k(x,y)|^2\\
&\le c_3\|f\|_\infty^2 \sum_{y\in B_k}\frac{1}{(1+|x|)^{2d+4}}\le c_4k^d(1+|x|)^{-2d-4}.
\end{split}\end{equation}

Fix $x\in \Z^d$. Note that $\Ee[\eta_k(x,z)]=0$ and $\{\eta_k(x,z)\}_{z\neq x}$ is a sequence of independent random variables.
Combining this with
\eqref{l3-1-3a-}
and the Hoeffding inequality (applying \cite[Theorem 2.16 on p.\ 21]{BDR},
taking $Z_{n}=\eta_k(x,y)$ and letting $n \to \infty$ in (2.31) there)
 yields that for every $\kappa>0$,
\begin{align*}
\forall x\in B_{2k}, \quad \Pp\left(\left|\sum_{z\in \Z^d}\eta_k(x,z)\right|>\kappa k^{-2}\log^{1/2} k\right)
 &\le 2\exp\left(- 2\left(\frac{\kappa k^{-2}\log^{1/2} k}{\|b_k(x)\|_2}\right)^2\right)\\
& \le
 2 \exp\left(-\frac{2\kappa^2}{c_2}\log k\right);
\end{align*}
\begin{align*}
\forall x\in B_{2k}^c, \quad &\Pp\left(\left|\sum_{z\in \Z^d}\eta_k(x,z)\right|>\kappa  k^{d/2}\frac{\log^{1/2}(1+|x|)}{(1+|x|)^{d+2}}\right)\le 2\exp\left(-\frac{2\kappa^2}{c_4}\log (1+|x|)\right).
\end{align*}
Hence, for any $\kappa> \max(\sqrt{c_2d}, \sqrt{c_4d})$ (noting that the choice of $\kappa$ may depend on $f$ since the constants $c_2$ and $c_4$ here depend on $f$), it holds that
\begin{align*}
&\sum_{k=1}^\infty \Bigg[\sum_{x\in B_{2k}}\Pp\left(\left|\sum_{z\in \Z^d}\eta_k(x,z)\right|>k^{-2}\log^{1/2} k\right)+\sum_{x\in B_{2k}^c}\Pp\left(\left|\sum_{z\in \Z^d}\eta_k(x,z)\right|>\kappa k^{d/2}\frac{\log^{1/2}(1+|x|)}{(1+|x|)^{d+2}}\right)\Bigg]<\infty.
\end{align*}
This along with Borel-Cantelli's lemma in turn yields that we can find a random variable $k_0(\w)\ge1$ such that for all $k\ge k_0(\w)$,
$$
\left|\sum_{y\in \Z^d}\eta_k(x,y)\right|\le
\begin{cases}
 \kappa k^{-2}\log^{1/2} k,&\quad x\in B_{2k},\\
 \kappa k^{d/2}\frac{\log^{1/2}(1+|x|)}{(1+|x|)^{d+2}},&\quad x\in B_{2k}^c.
\end{cases}
$$
Putting this into \eqref{l3-1-5}, we can get the desired assertion \eqref{l3-1-1}.
\end{proof}

\section{Proof of Theorem \ref{T:main}}
This section
is devoted to the proof of Theorem \ref{T:main}.
Recall that $f\in \mathscr{S}^\lambda_0$ and
\begin{align*}
    \bar u: =\bar R_\lambda f = (\lambda-\bar \sL)^{-1}f \in C_c^\infty(\R^d).
\end{align*}
For any $k\ge1$, by the mean value theorem,
\begin{equation}\label{e:pppqqq}\begin{split} \sum_{z\in k^{-1}\Z^d}\int_{\prod_{1\le i\le d} (z_i,z_i+k^{-1}]}
|\bar u(x)-\bar u(z)|^2\,dx
&\le   k^{-2} \sum_{z\in k^{-1}\Z^d}\int_0^1\int_{\prod_{1\le i\le d} (z_i,z_i+k^{-1}]}|\nabla \bar u(x+s(z-x))|^2\,dx\,ds\\
&\le c_0 k^{-2}, \end{split}\end{equation} where $c_0>0$ is independent of $k$.

\begin{proof}[Proof of Theorem $\ref{T:main}$]
We define the shorthand notations
\begin{align*}
    \bar u:= \bar R_\lambda f, \qquad u_k:=R^{(k), \w}_\lambda f.
\end{align*}
Without loss of generality, we assume that $ {\rm supp} [\bar u] \subset B_1$. We will divide the proof into six steps.

 \smallskip

{\bf Step 1:} For any $k\in \N_+$, let $m\ge0$ be the unique integer such that $2^m\le k<2^{m+1}$. Let
$n_0\ge 2$ be an integer such that $2^{n_0-2}\ge\kappa_0$, where $\kappa_0\ge1$ is the  constant given in \eqref{l2-1-1}.

Recall that $\mu$ denotes the counting measure on $\Z^d$. Let $\phi_m:B_{2^m}\to \R^d$ be the local corrector on $B_{2^m}$ which satisfies \eqref{e2-1}.
Define
\begin{equation*}
\hat \phi_{m+n_0}(z):=\phi_{m+n_0}(z)-\oint_{B_{2^{m+2}}}\phi_{m+n_0}d\mu,\quad z\in B_{2^{m+n_0}},
\end{equation*}
and extend it to $\hat \phi_{m+n_0}:\R^d \to \R^d$ by setting $\hat \phi_{m+n_0}(x)=\hat \phi_{m+n_0}(z)$ when $x\in \prod_{1\le i\le d}(z_i,z_i+1]\cap B_{2^{m+n_0}}$ for
the unique
$z\in \Z^d\cap B_{2^{m+n_0}}$, and $\phi_m(x)=0$ when $x\notin B_{2^{m+n_0}}$. Then, there exists $m_1:=m_1(\w)\ge 1$ such that for every $m\ge m_1$,
\begin{equation}\label{t3-1-2}
\begin{split}
\oint_{B_{2^{m+2}}}\hat \phi_{m+n_0}^2d\mu&=
\mu(B_{2^{m+2}})^{-1}\int_{B_{2^{m+2}}}\left(\phi_{m+n_0}(x)-\oint_{B_{2^{m+2}}}\phi_{m+n_0}\right)^2\mu(dx)\\
&\le c_12^{-m(d-2)}m^{-1}\E^\w_{B_{\kappa_02^{m+2}}}(\phi_{m+n_0},\phi_{m+n_0})\\
&\le c_12^{-m(d-2)}m^{-1}\E^\w_{B_{2^{m+n_0}}}(\phi_{m+n_0},\phi_{m+n_0})\le c_22^{2m}m^{-\frac{d-3-\gamma}{d-2}},
\end{split}
\end{equation}
where the first inequality is due to \eqref{l2-1-1}, the second inequality follows from the fact
$2^{n_0-2}\ge \kappa_0$, and in the last inequality we used \eqref{p2-1-2}.

Define the two-scale expansion
\begin{equation}\label{eq:v_k}
v_k(x):=\bar u(x)+k^{-1}
\langle \nabla \bar u(x),\hat \phi_{m+n_0}(k x)\rangle,
\quad x\in \R^d.
\end{equation}
Note that $v_k(x)=0$ for all $x\notin B_1$ due to the fact $\text{supp} [\bar u ] \subset B_1$. Hence, for any $k\ge 2^m$ with $m\ge m_1$,
\begin{equation}\label{t3-1-3}
\begin{split}
\int_{\R^d}|v_k(x)-\bar u(x)|^2 \,dx&=k^{-2}
\int_{B_{1}}
|\langle\nabla \bar u(x), \hat \phi_{m+n_0}(k x)\rangle|^2
\, dx \\
&\le c_3 k^{-2}\int_{B_{1}}|\hat \phi_{m+n_0}(k x)|^2\,dx=
c_3k^{-(2+d)}\int_{B_{k}}|\hat \phi_{m+n_0}(x)|^2\,dx\\
&\le c_3k^{-(2+d)}\int_{B_{2^{m+1}}}|\hat \phi_{m+n_0}(x)|^2\,\mu(dx)
\le c_4
\left(\log k\right)^{-\frac{d-3-\gamma}{d-2}},
\end{split}
\end{equation}
where the last inequality follows from \eqref{t3-1-2} and the fact that $2^m\le k<2^{m+1}$.
Furthermore, we restrict $v_k:\R^d \to \R$ to $v_k:k^{-1}\Z^d \to \R$. Applying the operator
$\sL^{(k)}$ to $v_k$, we get that
\begin{equation}\label{t3-1-3b}
\begin{split}
&\sL^{(k)}v_k(x)=\sL^{(k)}\bar u(x)\\
&\quad+(k \log k)^{-1}
\left\langle\hat \phi_{m+n_0}(kx),  \int_{k^{-1}\Z^d} \left(\nabla \bar u(x+z)-\nabla \bar u(x)\right)\frac{w_{kx,k(x+z)}}{|z|^{d+2}} \,\mu^{(k)}(dz) \right\rangle\\
&\quad+(k \log k)^{-1}\left\langle \nabla \bar u(x),  \int_{k^{-1}\Z^d}(\hat \phi_{m+n_0}\left(k(x+z)\right)-\hat \phi_{m+n_0}\left(k x\right))
\frac{w_{kx,k(x+z)}}{|z|^{d+2}}\,\mu^{(k)}(dz) \right\rangle\\
&\quad +(k \log k)^{-1}\int_{k^{-1}\Z^d}\left\langle \nabla \bar u(x+z)-\nabla \bar u(x)
, \hat \phi_{m+n_0}(k(x+z))-\hat \phi_{m+n_0}(kx)\right\rangle\frac{w_{kx,k(x+z)}}{|z|^{d+2}}\,\mu^{(k)}(dz)\\
&=:\sum_{i=1}^4 I_i^{(k)}(x).
\end{split}
\end{equation}
Steps~2-5 are devoted to the terms $I_i^{(k)}, i=1,2,3,4$.

\smallskip

{\bf Step 2:}
By \eqref{e:operator}, \eqref{e:operator0} and \eqref{e2-2},
\begin{equation}\label{t3-1-3c}
\begin{split}
I_1^{(k)}(x)&=\widehat \sL^{(k)} \bar u(x)+
(\log k)^{-1}\left\langle\nabla \bar u(x), \int_{k^{-1}\Z^d}z\frac{w_{kx,k(x+z)}}{|z|^{d+2}}\,\mu^{(k)}(dz) \right\rangle\\
&= \widehat \sL^{(k)} \bar u(x)+k (\log k)^{-1}
\left\langle\nabla \bar u(x),  \int_{\Z^d}z\frac{w_{k x,k x+z}}{|z|^{d+2}}\,\mu(dz) \right\rangle\\
&= \widehat \sL^{(k)} \bar u(x)+k (\log k)^{-1}\langle \nabla \bar u(x), V(k x)\rangle.
\end{split}
\end{equation}
According to
Propositions \ref{l3-1-0} and \ref{l3-1},
we can find a random  variable $m_2:=m_2(\w)\ge1$ such that for every $k\ge 2^m$ with $m\ge m_2 $,
$$
\int_{k^{-1}\Z^d}|\widehat \sL^{(k)}\bar u(x)-\bar \sL \bar u(x)|^2\, \mu^{(k)}(dx)\le \frac{c_5}{\log k}.
$$
Therefore, we obtain
\begin{equation}\label{t3-1-3a}
I_1^{(k)}(x)=\bar \sL \bar u(x)+k(\log k)^{-1}
\left\langle\nabla \bar u(x), V(k x)\right\rangle
+K_1^{(k)}(x),
\end{equation}
where for all $k\ge2^{m_2}$,
$$
\int_{k^{-1}\Z^d}|K_1^{(k)}(x)|^2 \,\mu^{(k)}(dx)\le \frac{c_5}{\log k}.
$$

\smallskip

{\bf Step 3:}
We set
\begin{align*}
I_2^{(k)}(x)&=(k \log k)^{-1}
\Bigg\langle\hat \phi_{m+n_0}(k x),
\int_{k^{-1}\Z^d} \left(\nabla \bar u(x+z)-\nabla \bar u(x)-\langle \nabla^2 \bar u(x), z\I_{\{|z|\le 1\}}\rangle\right)\frac{w_{kx,k(x+z)}}{|z|^{d+2}}\,\mu^{(k)}(dz)\\
&\qquad\qquad\qquad\qquad\qquad\qquad+\int_{k^{-1}\Z^d}\left\langle\nabla^2 \bar u(x), z\I_{\{|z|\le 1\}}\right\rangle\frac{w_{kx,k(x+z)}}{|z|^{d+2}}\,
\mu^{(k)}(dz)\Bigg\rangle\\
&=:I_{21}^{(k)}(x)+I_{22}^{(k)}(x).
\end{align*}

For all $x\in B_1$, by the mean-value theorem,
\begin{align*}
&\left|\int_{k^{-1}\Z^d} \left(\nabla \bar u(x+z)-\nabla \bar u(x)-\langle \nabla^2 \bar u(x),
z\I_{\{|z|\le 1\}}\rangle\right)\frac{w_{kx,k(x+z)}}{|z|^{d+2}}\,\mu^{(k)}(dz)\right|\\
&\le \frac{\|\nabla^2 \bar u\|_\infty}{2}\cdot \left(\int_{\{z\in k^{-1}\Z^d:|z|\le 1\}}|z|^{-d} \mu^{(k)}(dz)\right)
+2\|\bar u\|_\infty\cdot\left(\int_{\{z\in k^{-1}\Z^d:|z|>1\}}|z|^{-d-2}\mu^{(k)}(dz)\right)\\
&\le c_6\log k,
\end{align*}
where in the last inequality we used the fact that
\begin{align}\label{t3-1-4a}
\int_{\{z\in k^{-1}\Z^d:|z|\le 1\}}|z|^{-d} \mu^{(k)}(dz)\le c_7\log k.
\end{align}
On the other hand, since ${\rm supp}[\bar u]\subset B_1$, $u(x+z)\neq 0$ only if $z\in B_1(x)$, and so for all $x\in B_1^c$,
\begin{align*}
&\left|\int_{k^{-1}\Z^d} \left(\nabla \bar u(x+z)-\nabla \bar u(x)-\langle \nabla^2 \bar u(x),
z\I_{\{|z|\le 1\}}\rangle\right)\frac{w_{kx,k(x+z)}}{|z|^{d+2}}\,\mu^{(k)}(dz)\right|\\
&\le \|\bar u\|_\infty \cdot\left(\int_{B_1(x)}|z|^{-d-2}\mu^{(k)}(dz)\right)\\
&\le c_8(1+|x|)^{-d-2}.
\end{align*}
Thus,
 $$|I_{21}^{(k)}(x)|\le c_9|\hat \phi_{m+n_0}(k x)|\left(k^{-1}\I_{B_1}(x)+(k\log k)^{-1}\frac{1}{(1+|x|)^{d+2}}\I_{B_1^c}(x)\right),$$
which implies that for all $y\in B_{1/2}$,
 \begin{align*} |I_{21}^{(k)}(x)|\le &c_{10}
 k^{-1}
 |\hat \phi_{m+n_0}(k x)|\I_{B_1}(x)+ c_{10}(k \log k)^{-1} \frac{|\hat \phi_{m+n_0}(k x)-\hat \phi_{m+n_0}(k y)|}{|x-y|^{d+2}}\I_{B_1^c}(x)\\
 &+c_{10} (k \log k)^{-1} \frac{| \hat \phi_{m+n_0}(k y)|}{(1+|x|)^{d+2}}\I_{B_1^c}(x).\end{align*}
 Hence,
\begin{align*}
\int_{k^{-1}\Z^d} |I_{21}^{(k)}(x)|^2\mu^{(k)}(dx)& \le
c_{11}k^{-2}
\oint_{B_{2^{m+1}}}|\hat \phi_{m+n_0}(x)|^2\,\mu(dx)\\
&\quad+ \frac{c_{11}(k\log k)^{-2}}{\mu^{(k)}(B_{1/2})}\int_{k^{-1}B_{2^{m+n_0}}\setminus B_1}\int_{B_{1/2}} \frac{|\hat \phi_{m+n_0}( k x)-\hat \phi_{m+n_0}( k y)|^2}{|x-y|^{2(d+2)}}\,\mu^{(k)}(dy)\,\mu^{(k)}(dx)\\
&\quad+\frac{c_{11}(k\log k)^{-2}}{\mu^{(k)}(B_{1/2})}\int_{B_{1/2}}|\hat \phi_{m+n_0}( k y)|^2\,\mu^{(k)}(dy)\int_{k^{-1}B_{2^{m+n_0}}}\frac{1}{(1+|x|)^{2(d+2)}}\,\mu^{(k)}(dx)\\
&=:I_{211}^{(k)}+I_{212}^{(k)}+I_{213}^{(k)}. \end{align*}

By \eqref{t3-1-2} and the fact that $2^m\le k<2^{m+1}$, it holds
that for any $k\ge2^{m_1}$,
\begin{equation*}
I_{211}^{(k)}\le
c_{12}
(\log k)^{-\frac{d-3-\gamma}{d-2}}.
\end{equation*}
According to \eqref{p2-1-2},  for any $k\ge2^{m_1}$,
\begin{align*}I_{212}^{(k)}\le& c_{13}(\log k)^{-2}k^{-d}\int_{ B_{2^{m+n_0}}}\int_{ B_{2^{m+n_0}}} \frac{|\hat \phi_{m+n_0}(  x)-\hat \phi_{m+n_0}(  y)|^2}{|x-y|^{d+2}}\,\mu(dy)\,\mu(dx)\\
\le&c_{14}(\log k)^{-2+\frac{1+\gamma}{d-2}}.\end{align*} On the other hand, it is easy to see from \eqref{t3-1-2} that  for any $k\ge2^{m_1}$,
$$I_{213}^{(k)}\le c_{15}
(\log k)^{-2-\frac{d-3-\gamma}{d-2}}.$$
Thus, by all these estimates above, we obtain that for any $k\ge2^{m_1}$,
$$I_{21}^{(k)}\le c_{16}
(\log k)^{-\frac{d-3-\gamma}{d-2}}.
$$

For every $g\in L^2(k^{-1}\Z^d;\mu^{(k)})$,  we have
\begin{align*}
&\int_{k^{-1}\Z^d}I_{22}^{(k)}(x)g(x)\,\mu^{(k)}(dx)\\
&=(k\log k)^{-1}\int_{k^{-1}\Z^d}\int_{\{|z|\le 1\}}
\langle G(x),\hat \phi_{m+n_0}(kx)\otimes z\rangle
\frac{w_{k x, k(x+z)}}{|z|^{d+2}}\,\mu^{(k)}(dz)\,\mu^{(k)}(dx)\\
&=(k\log k)^{-1}\Bigg(\frac{1}{2}\int_{k^{-1}\Z^d}\int_{\{|z|\le 1\}}
\langle G(x),\hat \phi_{m+n_0}(kx)\otimes z\rangle
\frac{w_{k x, k(x+z)}}{|z|^{d+2}}\mu^{(k)}(dz)\,\mu^{(k)}(dx)\\
&\qquad\qquad\qquad\,\, -\frac{1}{2}\int_{k^{-1}\Z^d}\int_{\{|z|\le 1\}}
\langle G(x+z),\hat \phi_{m+n_0}(k(x+z))\otimes z\rangle
\frac{w_{k x, k(x+z)}}{|z|^{d+2}}\,\mu^{(k)}(dz)\,\mu^{(k)}(dx)\Bigg)\\
&=(2k\log k)^{-1}\int_{k^{-1}\Z^d}\\
&\quad\times \int_{\{|z|\le 1\}}
\left\langle G(x),\hat \phi_{m+n_0}(kx)\otimes z\right\rangle- \left\langle G(x+z),\hat \phi_{m+n_0}(k(x+z))\otimes z\right\rangle
\frac{w_{k x, k(x+z)}}{|z|^{d+2}}\,\mu^{(k)}(dz)\,\mu^{(k)}(dx),
\end{align*}
where $G(x):=g(x)\nabla^2 \bar u(x)$, and in the second equality we have used the change of variables $x=\tilde x+\tilde z$ and $z=-\tilde z$
as well as the symmetric property $w_{x,y}=w_{y,x}$ for all $x,y\in \Z^d$.
Next, we set
\begin{align*}
&\left|\int_{k^{-1}\Z^d}I_{22}^{(k)}(x)g(x)\,\mu^{(k)}(dx)\right|\\
&\le (2k\log k)^{-1}\Bigg(\int_{B_{2}}\int_{\{|z|\le 1\}}|\hat \phi_{m+n_0}(k x)|
|G(x+z)-G(x)||z|\frac{w_{k x,k(x+z)}}{|z|^{d+2}}\,\mu^{(k)}(dz)\,\mu^{(k)}(dx)\\
&\qquad \qquad\qquad\quad +\int_{B_1}\int_{\{|z|\le 1\}}|G(x)|
|\hat \phi_{m+n_0}(k (x+z))-\hat \phi_{m+n_0}(k x)||z|\frac{w_{k x,k(x+z)}}{|z|^{d+2}}\,\mu^{(k)}(dz)\,\mu^{(k)}(dx)\Bigg)\\
&=:I_{221}^{(k)}+I_{222}^{(k)},
\end{align*}
where we used the change of variables $x=\tilde x+\tilde z$ and $z=-\tilde z$ as before,
and also used the fact that $\text{supp} [G] \subset B_1$.

For any
$G=\{G^{(ij)}\}_{1\le i,j\le d}:k^{-1}\Z^d \to \R^{d}\times \R^d$
with $G^{(ij)}\in L^2(k^{-1}\Z^d;\mu^{(k)})$ for every $1\le i,j \le d$,
let
\begin{align*}
   \mathscr{E}^{(k),\w}(G,G)
 &:
 =\frac{1}{2\log k}
 \sum_{i,j=1}^d
 \int_{k^{-1}\Z^d}\int_{k^{-1}\Z^d}
\frac{(G^{(ij)}(x+z)-G^{(ij)}(x))^2w_{k x,k(x+z)}(\w)}{|z|^{d+2}}\,\mu^{(k)}(dz)\,\mu^{(k)}(dx).
\end{align*}
Then, we have for all $k\ge2^{m_1}$,
\begin{align*}
I_{221}^{(k)} &\le c_{17}(k\log k)^{-1}
\left(\int_{B_2}|\hat \phi_{m+n_0}(k x)|^2\left(\int_{\{|z|\le 1\}}\frac{|z|^2}{|z|^{d+2}}\,\mu^{(k)}(dz)\right)
\,\mu^{(k)}(dx)\right)^{1/2}\\
&\quad\times\left(\int_{k^{-1}\Z^d}\int_{\{|z|\le 1\}}\frac{|G(x+z)-G(x)|^2w_{k x,k(x+z)}}{|z|^{d+2}}
\,\mu^{(k)}(dz)\,\mu^{(k)}(dx)\right)^{1/2}\\
&\le c_{18}k^{-1}\left(\int_{B_2}|\hat \phi_{m+n_0}(k x)|^2\,\mu^{(k)}(dx)\right)^{1/2}  \mathscr{E}^{(k),\w}(G,G) ^{1/2}\\
&\le c_{19}
(\log k)^{-\frac{d-3-\gamma}{2(d-2)}}
\,\, \mathscr{E}^{(k),\w}(G,G) ^{1/2}\\
&\le c_{20}
(\log k)^{-\frac{d-3-\gamma}{2(d-2)}}
\left( \mathscr{E}^{(k),\w}(g,g)^{1/2}+\|g\|_{L^2(k^{-1}\Z^d;\mu^{(k)})}\right),
\end{align*}
where the first inequality is due to the Cauchy-Schwartz inequality, the second inequality follows from
\eqref{t3-1-4a},
in the third inequality we used
\eqref{t3-1-2} and $2^m\le k <2^{m+1}$, and the last inequality follows from some direct computations and
the property that $w_{x,y}\le C_2$ for all $x,y\in \Z^d$.

On the other hand, by the Cauchy-Schwartz inequality again,
\begin{equation}\label{t3-1-4}
\begin{split}
I_{222}^{(k)}
&\le c_{21}(k\log k)^{-1}\left(\int_{B_1}G^2(x)\left(\int_{\{|z|\le 1\}}\frac{|z|^2}{|z|^{d+2}}\,\mu^{(k)}(dz)\right)\,
\mu^{(k)}(dx)\right)^{1/2}\\
&\quad\times\left(\int_{B_1}\int_{\{|z|\le 1\}}
\left|\hat \phi_{m+n_0}(k(x+z))-\hat \phi_{m+n_0}(k x)\right|^2
\frac{w_{k x,k(x+z)}}{|z|^{d+2}}\,\mu^{(k)}(dz)\,\mu^{(k)}(dx)\right)^{1/2}.
\end{split}
\end{equation}
Note that for all $k\ge2^{m_1}$,
\begin{align*}
&\int_{B_1}\int_{\{|z|\le 1\}}
\left|\hat \phi_{m+n_0}(k(x+z))-\hat \phi_{m+n_0}(k x)\right|^2
\frac{w_{k x,k(x+z)}}{|z|^{d+2}}\,\mu^{(k)}(dz)\,\mu^{(k)}(dx)\\
&\le k^{-(d-2)}
\int_{B_{k}}\int_{B_{2k}}\left|\hat \phi_{m+n_0}(x)-\hat \phi_{m+n_0}(y)\right|^2
\frac{w_{x,y}}{|x-y|^{d+2}}\,\mu(dy)\,\mu(dx)\\
&\le k^{-(d-2)}
\int_{B_{2k}}\int_{B_{2k}}\left|\hat \phi_{m+n_0}(y)-\hat \phi_{m+n_0}(x)\right|^2
\frac{w_{x,y}}{|x-y|^{d+2}}\,
\mu(dy)\,\mu(dx)\\
&\le k^{-(d-2)}
\mathscr{E}_{B_{2^{m+n_0}}}^\w(\phi_{m+n_0},\phi_{m+n_0})
\le c_{22}k^{2}
\left(\log k\right)^{\frac{1+\gamma}{d-2}},
\end{align*}
where we used \eqref{p2-1-2} and the fact
\begin{align*}
\mathscr{E}_{B_{2k}}^\w(\hat \phi_{m+n_0},\hat \phi_{m+n_0})=\mathscr{E}_{B_{2k}}^\w(\phi_{m+n_0},\phi_{m+n_0})
\le \mathscr{E}_{B_{2^{m+n_0}}}^\w(\phi_{m+n_0},\phi_{m+n_0}).
\end{align*}
Combining this with \eqref{t3-1-4a} into \eqref{t3-1-4} yields that for all $k\ge2^{m_1}$,
\begin{align*}
I_{222}^{(k)} \le c_{23}
(\log k)^{-\frac{d-3-\gamma}{2(d-2)}}
\|G\|_{L^2((B_1;\R^d\times \R^d);\mu^{(k)})}\le c_{24}
(\log k)^{-\frac{d-3-\gamma}{2(d-2)}}
\|g\|_{L^2(B_1;\mu^{(k)})}.
\end{align*}

According to
all the estimates above, we have
\begin{equation}\label{t3-1-5}
I_2^{(k)}(x)=K_{21}^{(k)}(x)+K_{22}^{(k)}(x),
\end{equation}
where for every $k\ge2^{m_1}$,
$$\int_{k^{-1}\Z^d}|K_{21}^{(k)}(x)|^2\,\mu^{(k)}(dx)\le c_{25}
(\log k)^{-\frac{d-3-\gamma}{d-2}},
$$ and
\begin{align*}
\left|\int_{k^{-1}\Z^d}K_{22}^{(k)}(x)g(x)\,\mu^{(k)}(dx)\right|\le c_{25}
(\log k)^{-\frac{d-3-\gamma}{2(d-2)}}
\left(\E^{(k),\w}(g,g)^{1/2}+\|g\|_{L^2(k^{-1}\Z^d;\mu^{(k)})}\right )
\end{align*}
for all $g\in L^2(k^{-1}\Z^d;\mu^{(k)})$.

\smallskip

{\bf Step 4:}
We know that, by the change of variables,
\begin{align*}
I_3^{(k)}(x)&=k(\log k)^{-1}
\left\langle\nabla \bar u(x),  \int_{\Z^d}\left(\hat \phi_{m+n_0}(y)-\hat \phi_{m+n_0}(k x)\right)\frac{w_{kx,y}}{|y-k x|^{d+2}}\,\mu(dy) \right\rangle\\
&=k(\log k)^{-1}\Bigg\langle \nabla \bar u(x),  \int_{B_{2^{m+n_0}}}\left(\hat \phi_{m+n_0}(y)-\hat \phi_{m+n_0}(k x)\right)\frac{w_{kx,y}}
{|y-k x|^{d+2}}\,\mu(dy)\\
&\qquad\qquad\qquad \qquad\quad+\int_{B_{2^{m+n_0}}^c}\left(\hat \phi_{m+n_0}(y)-\hat \phi_{m+n_0}(k x)\right)\frac{w_{kx,y}}
{|y-k x|^{d+2}}\,\mu(dy)\Bigg\rangle\\
&=k(\log k)^{-1}\left\langle \nabla \bar u(x),\sL_{B_{2^{m+n_0}}}\phi_{m+n_0}(k x)+\int_{B_{2^{m+n_0}}^c}\left(\hat \phi_{m+n_0}(y)-\hat \phi_{m+n_0}(k x)\right)
\frac{w_{kx,y}}{|y-k x|^{d+2}}\,\mu(dy) \right\rangle\\
&=k(\log k)^{-1}\left\langle\nabla \bar u(x),-V(k x)+\oint_{B_{2^{m+n_0}}}V\,d\mu+\int_{B_{2^{m+n_0}}^c}\left(\hat \phi_{m+n_0}(y)-\hat \phi_{m+n_0}(k x)\right)\frac{w_{kx,y}}{|y-k x|^{d+2}}\,\mu(dy) \right\rangle\\
&=:-k(\log k)^{-1}\langle\nabla \bar u(x), V(k x)\rangle+I_{31}^{(k)}(x)+I_{32}^{(k)}(x),
\end{align*}
where in the fourth equality we used \eqref{e2-1} for the corrector $\phi_{m+n_0}$ that implies
$$\sL_{B_{2^{m+n_0}}}\phi_{m+n_0}(kx)=\sL_{B_{2^{m+n_0}}}\hat \phi_{m+n_0}(kx),\quad x\in B_1\subset B_{k^{-1}2^{m+n_0}}.$$

Combining \eqref{p2-1-4} with the fact that $2^m\le k<2^{m+1}$, for every
$\theta\in (2/d,1)$,
we can find a random variable $m_3:=m_3(\w)>1$ such that
for all $k\ge2^{m_3}$,
$$
\int_{k^{-1}\Z^d}|I_{31}^{(k)}(x)|^2\,\mu^{(k)}(dx) \le k^2(\log k)^{-2}\|\nabla \bar u\|_\infty^2\int_{B_1}\Big|\oint_{B_{2^{m+n_0}}}V\,d\mu\Big|^2 \,
\mu^{(k)}(dx)
\le c_{26}k^{2-\theta d}(\log k)^{-2}.
$$
Since $\hat \phi_{m+n_0}(y)=0$ for all $y\in B_{2^{m+n_0}}^c$ and
$\text{supp} [ u ] \subset B_1$, it holds that for all $k\ge2^{m_1}$,
\begin{align*}
\int_{k^{-1}\Z^d}|I_{32}^{(k)}(x)|^2\,\mu^{(k)}(dx)&\le
k^2(\log k)^{-2}\|\nabla \bar u\|_\infty^2
\int_{B_1}|\hat \phi_{m+n_0}(k x)|^2 \Bigg(\sum_{y\in k^{-1}\Z^d:|y-k x|\ge k}\frac{w_{kx,y}}{|y-k x|^{d+2}}\Bigg)^2\,
\mu^{(k)}(dx)\\
&\le c_{27}k^{-2}(\log k)^{-2}\int_{B_1}|\hat \phi_{m+n_0}(k x)|^2\,\mu^{(k)}(dx)\le c_{28}
(\log k)^{-2-\frac{d-3-\gamma}{d-2}},
\end{align*}
where in the last inequality is due to \eqref{t3-1-2} again.

Summarizing all the estimates above, we
conclude that
\begin{equation}\label{t3-1-6}
\begin{split}
I_3^{(k)}(x)&=-k(\log k)^{-1}
\langle \nabla \bar u(x), V(kx)\rangle
+K_3^{(k)}(x),
\end{split}
\end{equation}
where for all $k\ge2^{\max\{m_1, m_3\}}$, $$\int_{k^{-1}\Z^d}|K_3^{(k)}(x)|^2\,\mu^{(k)}(dx)\le c_{29}
(\log k)^{-2-\frac{d-3-\gamma}{d-2}}.
$$
Here,  we used the fact that
$2-\theta d<0$
as
$\theta\in (2/d,1)$.
In particular, from \eqref{t3-1-3a} and \eqref{t3-1-6} one can see that the term
$k (\log k)^{-1}V(kx)$ in the expression \eqref{t3-1-3a} for $\sL^{(k)}v_k(x)$ (which may blow up
as $k \to \infty$) will vanish.

\smallskip

{\bf Step  5:}
For the estimate for $I_4^{(k)}(x)$, we consider  two different cases $x\in B_{k^{-1}2^{m+n_0}}$ and $x\notin B_{k^{-1}2^{m+n_0}}$.
When $x\notin B_{k^{-1}2^{m+n_0}}$, by the fact $\nabla \bar u(z)\neq 0$ only if $z\in B_1$, we obtain that
for every $k\ge2^{m_1}$,
\begin{align*}
|I_4^{(k)}(x)|&=(k\log k)^{-1}\left|\int_{\{x+z\in B_1\}}\frac{
\left\langle \nabla \bar u(x+z), \hat \phi_{m+n_0}\left(k(x+z)\right)\right\rangle
w_{k x,k(x+z)}}{|z|^{d+2}}\,\mu^{(k)}(dz)\right|\\
&=(k\log k)^{-1}\left|\int_{B_1}\frac{
\left\langle \nabla \bar u(y),
\hat \phi_{m+n_0}\left(k y\right)\right\rangle
w_{k x,k y}}{|y-x|^{d+2}}\,\mu^{(k)}(dy)\right|\\
&\le c_{30}(k\log k)^{-1}(1+|x|)^{-d-2}
\left(\int_{ B_1}\left|\hat \phi_{m+n_0}\left(k y\right)\right|\mu^{(k)}(dy)\right)\\
&\le c_{31}(k\log k)^{-1}(1+|x|)^{-d-2}
 \left(\int_{B_1}\left|\hat \phi_{m+n_0}\left(k y\right)\right|^2\mu^{(k)}(dy)\right)^{1/2}\\
&\le c_{32}
(\log k)^{-1-\frac{d-3-\gamma}{2(d-2)}}
\,(1+|x|)^{-d-2},
\end{align*}
where the first equality follows from the fact that $\hat \phi_{m+n_0}(z)=0$ for all $z\in B_{2^{m+n_0}}^c$ and so $\hat \phi_{m+n_0}\left(k x\right)=0$ for all $x\in B_{k^{-1}2^{m+n_0}}^c$,
the first inequality
follows from the fact that $|y-x|\ge c_{33}(1+|x|)$ for all $y\in B_1$ and $x\in B_{k^{-1}2^{m+n_0}}^c$,
and in the last inequality we used \eqref{t3-1-2}.

When $x\in B_{k^{-1}2^{m+n_0}}$, it holds that
\begin{align*}
|I_4^{(k)}(x)|&=(k\log k)^{-1}\Bigg|
\int_{k^{-1}B_{{2^{m+n_0}}}}
\left\langle \nabla \bar u(y)-\nabla \bar u(x), \hat \phi_{m+n_0}\left(k y\right)-\hat \phi_{m+n_0}\left(k x\right)\right\rangle
\frac{w_{k x,k y}}{|y-x|^{d+2}}\,\mu^{(k)}(dy)\\
&\qquad \qquad\qquad+
\langle \nabla \bar u(x), \hat \phi_{m+n_0}(k x)\rangle
\int_{B_{k^{-1}{2^{m+n_0}}}^c}
\frac{w_{k x,k y}}{|y-x|^{d+2}}\,\mu^{(k)}(dy)\Bigg|\\
&\le c_{34}(k\log k)^{-1}\left(\int_{B_{k^{-1}2^{m+n_0}}}
\frac{\left|\nabla \bar u\left(y\right)-\nabla \bar u\left(x\right)\right|^2}{|y-x|^{d+2}}\,\mu^{(k)}(dy)\right)^{1/2}\\
&\qquad\times
\left(\int_{B_{k^{-1}2^{m+n_0}}}
\frac{\left|\hat \phi_{m+n_0}\left(k y\right)-\hat \phi_{m+n_0}\left(k x\right)\right|^2w_{k x,k y}}{|y-x|^{d+2}}\,\mu^{(k)}(dy)
\right)^{1/2}\\
&\quad +c_{34}(k\log k)^{-1}|\hat \phi_{m+n_0}(k x)|\I_{B_1}(x)\cdot\left(\int_{\{y\in k^{-1}\Z^d:|y-x|\ge 1/2\}}\frac{1}{|y-x|^{d+2}}\,\mu^{(k)}(dy)\right)\\
&\le
c_{35}k^{-1}(\log k)^{-1/2}\left(\int_{B_{k^{-1}2^{m+n_0}}}
\frac{\left|\hat \phi_{m+n_0}\left(k y\right)-\hat \phi_{m+n_0}\left(k x\right)\right|^2w_{k x,k y}}{|y-x|^{d+2}}\,\mu^{(k)}(dy)
\right)^{1/2}\\
&\quad +c_{35}(k \log k)^{-1}|\hat \phi_{m+n_0}(k x)|\I_{B_1}(x).
\end{align*}
Here the equality above
is due to the fact $\nabla \bar u(y)=\hat \phi_{m+n_0}(k y)=0$ for all $y\in B_{k^{-1}{2^{m+n_0}}}^c$,
in the first
inequality we used the Cauchy-Schwarz inequality, the fact $\nabla \bar u(x)\neq 0$ only if
  $x\in B_1$, as well as the property $|y-x|\ge 1/2$ for every $y\in B_{k^{-1}{2^{m+n_0}}}^c$ and $x\in B_{1}$,
and the last inequality follows from \eqref{t3-1-4a}.

Combining with both estimate above yields that for all $k\ge2^{m_1}$,
\begin{align*}
|I_4^{(k)}(x)|&\le c_{36}\Bigg(
k^{-1}(\log k)^{-1/2}
\left(\int_{B_{k^{-1}2^{m+n_0}}}
\frac{\left|\hat \phi_{m+n_0}\left(k y\right)-\hat \phi_{m+n_0}\left(k x\right)\right|^2w_{k x, k y}}{|y-x|^{d+2}}\,\mu^{(k)}(dy)
\right)^{1/2}\I_{B_{k^{-1}2^{m+n_0}}}(x)\\
&\qquad\quad +
(\log k)^{-1-\frac{d-3-\gamma}{2(d-2)}}
(1+|x|)^{-d-2}\I_{B_{k^{-1}2^{m+n_0}}^c}(x)
+(k\log k)^{-1}|\hat \phi_{m+n_0}(k x)|\I_{B_1}(x)\Bigg).
\end{align*}
Then, applying the Cauchy-Schwarz inequality again, we obtain that for all $k\ge 2^{m_1}$,
\begin{equation}\label{t3-1-7}
\begin{split}
&\int_{k^{-1}\Z^d}|I_4^{(k)}(x)|^2\,\mu^{(k)}(dx)\\
&\le
c_{37}
k^{-2}(\log k)^{-1}\int_{B_{k^{-1}B_{2^{m+n_0}}}}\int_{k^{-1}B_{2^{m+n_0}}}
\frac{\left|\hat \phi_{m+n_0}\left(k y\right)-\hat \phi_{m+n_0}\left(k x\right)\right|^2w_{k x,k y}}{|y-x|^{d+2}}\,\mu^{(k)}(dy)
\,\mu^{(k)}(dx)\\
&\quad +c_{37}
(\log k)^{-2-\frac{d-3-\gamma}{d-2}}
\int_{B_{k^{-1}2^{m+n_0}}^c}\frac{1}{(1+|x|)^{2d+4}}\,\mu^{(k)}(dx)
+c_{37}(k\log k)^{-2}\int_{B_{1}}|\hat \phi_{m+n_0}(k x)|^2\,\mu^{(k)}(dx)\\
&\le c_{38}k^{-d}(\log k)^{-
1}\mathscr{E}_{B_{2^{m+n_0}}}^\w(\phi_{m+n_0},\phi_{m+n_0})+c_{38}
(\log k)^{-2-\frac{d-3-\gamma}{d-2}}
\int_{B_{3/2}^c}\frac{1}{(1+|x|)^{2d+2\alpha}}\,\mu^{(k)}(dx)\\
&\quad +c_{38}(k\log k)^{-2}\oint_{B_{2^{m+1}}}|\hat \phi_{m+n_0}(x)|^2\,\mu(dx)\\
&\le c_{39}
(\log k)^{-\frac{d-3-\gamma}{d-2}},
\end{split}
\end{equation}
where the last inequality is due to \eqref{p2-1-2} and \eqref{t3-1-2}.

\smallskip

{\bf Step 6:}
Putting \eqref{t3-1-3a}, \eqref{t3-1-5}, \eqref{t3-1-6} and \eqref{t3-1-7} together, we find that
\begin{equation}\label{t3-1-8}
\begin{split}
\sL^{(k)}v_k(x)=\bar \sL \bar u(x)
+J_1^{(k)}(x)+J_2^{(k)}(x),\quad x\in k^{-1}\Z^d,
\end{split}
\end{equation}
where
\begin{equation}\label{t3-1-9}
\begin{split}
\left|\int_{k^{-1}\Z^d}J_1^{(k)}(x)g(x)\,\mu^{(k)}(dx)\right|&\le c_{40}
(\log k)^{-\frac{d-3-\gamma}{2(d-2)}}
\Big(\E^{(k),\w}(g,g)^{1/2}+\|g\|_{L^2(k^{-1}\Z^d;\,\mu^{(k)})}\Big ),\\
\int_{k^{-1}\Z^d}|J_2^k(x)|^2\,\mu^{(k)}(dx)&\le c_{40}
(\log k)^{-\frac{d-3-\gamma}{d-2}}
\end{split}
\end{equation} hold for all $k\ge k_0:=2^{\max\{m_1,m_2,m_3\}}$ and $g\in L^2(k^{-1}\Z^d;\mu^{(k)})$.
Therefore, by  \eqref{e3-2},
\begin{align*}
&\lambda(u_k(x)-v_k(x))-\sL^{(k)}(u_k-v_k)(x)\\
&= (\lambda u_k(x)-\sL^{(k)}u_k(x) )+
\lambda(\bar u(x)-v_k(x))- (\lambda \bar u(x)-
\bar \sL \bar u(x) )+J_1^{(k)}(x)+J_2^{(k)}(x)\\
&=J_1^{(k)}(x)+J_2^{(k)}(x)+\lambda (\bar u(x)-v_k(x))\\
&=:J_1^{(k)}(x)+Q^{(k)}(x),
\end{align*}
where in the second equality we used the fact that $(\lambda u_k-\sL^{(k)}u_k ) = (\lambda \bar u-
\bar \sL \bar u ) =f$.  We note that, according to \eqref{t3-1-3} and \eqref{t3-1-9}, for all
$k\ge k_0$,
\begin{equation}\label{t3-1-10}
\int_{k^{-1}\Z^d}|Q^{(k)}(x)|^2\,\mu^{(k)}(dx)\le c_{41}
(\log k)^{-\frac{d-3-\gamma}{d-2}}.
\end{equation}
Here and in what follows the constants involved in may depend on $\lambda$.

Since $u_k,v_k\in L^2(k^{-1}\Z^d;\mu^{(k)})$,
multiplying $u_k-v_k$
on
both sides of the equality above, integrating  with respect to $\mu^{(k)}$
and applying \eqref{e1-1a},
 we can find $k_0^*(\lambda)\ge1$ so that for all
$k\ge k^*_0(\lambda)$,
\begin{align*}
&\lambda\|u_k-v_k\|^2_{L^2(k^{-1}\Z^d;\mu^{(k)})}+ \E^{(k),\w}(u_k-v_k, u_k-v_k)\\
&=\lambda \|u_k-v_k\|^2_{L^2(k^{-1}\Z^d;\mu^{(k)})}-
\int_{k^{-1}\Z^d}\sL^{(k)}(u_k-v_k)(x)\cdot\left(u_k-v_k\right)(x)\,\mu^{(k)}(dx)\\
&=\int_{k^{-1}\Z^d}J_1^{(k)}(x)\left(u_k-v_k\right)(x)\,\mu^{(k)}(dx)+
\int_{k^{-1}\Z^d}Q^{(k)}(x)\cdot\left(u_k-v_k\right)(x)\,\mu^{(k)}(dx)\\
&\le \int_{k^{-1}\Z^d}J_1^{(k)}(x)\cdot\left(u_k-v_k\right)(x)\,\mu^{(k)}(dx)+2\lambda^{-1}
\int_{k^{-1}\Z^d}|Q^{(k)}(x)|^2\,\mu^{(k)}(dx)+\frac{\lambda}{2}\|u_k-v_k\|^2_{L^2(k^{-1}\Z^d;\mu^{(k)})}\\
&\le c_{42}(\log k)^{-\frac{d-3-\gamma}{2(d-2)}}
\big(\E^{(k),\w}(u_k-v_k, u_k-v_k)^{1/2}+\|u_k-v_k\|_{L^2(k^{-1}\Z^d;\mu^{(k)})}\big)\\
&\quad +c_{42}(\log k)^{-\frac{d-3-\gamma}{d-2}}
+\frac{\lambda}{2}\|u_k-v_k\|^2_{L^2(k^{-1}\Z^d;\mu^{(k)})}\\
&\le \frac{3\lambda}{4}\|u_k-v_k\|^2_{L^2(k^{-1}\Z^d;\mu^{(k)})}+\frac{3}{4}
\E^{(k),\w}(u_k-v_k, u_k-v_k)+c_{43}(\log k)^{-\frac{d-3-\gamma}{d-2}},
\end{align*}
where in the first and the last inequalities we used Young's inequality, and  the second inequality follows from
\eqref{t3-1-9} and \eqref{t3-1-10}.
From the estimate above, we arrive at that
\begin{align*}
\|u_k-v_k\|^2_{L^2(k^{-1}\Z^d;\mu^{(k)})}\le c_{44}(\log k)^{-\frac{d-3-\gamma}{d-2}}.
\end{align*}
Therefore, it holds that for all $k\ge k^*_0(\lambda) $,
\begin{equation}\label{t3-1-12}
\begin{split}
 \|u_k-v_k\|^2_{L^2(\R^d;dx)}
&=\sum_{z\in k^{-1}\Z^d}\int_{\prod_{1\le i\le d}(z_i,z_i+k^{-1}]}|u_k(x)-v_k(x)|^2\,dx\\
&\le 2\sum_{z\in k^{-1}\Z^d}\int_{\prod_{1\le i\le d}(z_i,z_i+k^{-1}]}\left(|u_k(x)-v_k(z)|^2+
|v_k(x)-v_k(z)|^2\right)\,dx\\
&=2\|u_k-v_k\|^2_{L^2(k^{-1}\Z^d;\mu^{(k)})}+2\sum_{z\in k^{-1}\Z^d}\int_{\prod_{1\le i\le d}(z_i,z_i+k^{-1}]}
|v_k(x)-v_k(z)|^2\,dx\\
&\le 2\|u_k-v_k\|^2_{L^2(k^{-1}\Z^d;\mu^{(k)})}+c_{45}\bigg(k^{-2}+k^{-2}\sum_{z\in k^{-1}\Z^d}k^{-d}|\hat \phi_{m+n_0}(k z)|^2\bigg)\\
&\le 2\|u_k-v_k\|^2_{L^2(k^{-1}\Z^d;\mu^{(k)})}+c_{46}\bigg(k^{-2}+k^{-2}\oint_{B_{2^{m+2}}}|\hat \phi_{m+n_0}(z)|^2\,\mu(dz)\bigg)\\
&\le c_{47}(\log k)^{-\frac{d-3-\gamma}{d-2}},
\end{split}
\end{equation}
where the second inequality
follows from the estimate below, due to \eqref{eq:v_k} and \eqref{e:pppqqq},
\begin{align*}
|v_k(x)-v_k(z)|&\le c_{48}\big(|\bar u(x)-\bar u(z)|+
k^{-1}|\nabla \bar u(x)-\nabla \bar u(z)|\cdot |\hat \phi_{m+n_0}(kz)|\big)\\
&\le c_{49}k^{-1}(1+|\hat \phi_{m+n_0}(kz)|),
\quad z\in k^{-1}\Z^d, x\in \Pi_{1\le i\le d}(z_i,z_i+k^{-1}] \cap B_1.
\end{align*}
Note that in the first inequality above we used the fact
$\hat \phi_{m+n_0}(kx)=\hat \phi_{m+n_0}(kz)$ for every  $x\in \Pi_{1\le i\le d}(z_i,z_i+k^{-1}]$
by the way of extending the function at the beginning of {\bf Step 1}.

Combining \eqref{t3-1-12} with \eqref{t3-1-3}, we can prove the desired conclusion.
\end{proof}

\ \

\noindent {\bf Acknowledgements.}\,\,
As we were finalizing this paper, we were informed by Professor Paul Dario that he with Professor Ahmed Bou-Rabee had completed work on the same topic. The overlap of our research directions serves as a valuable encouragement for our team. We are grateful for their independent efforts, which have inspired us and strengthened our confidence in this work. 
The research of Xin Chen is supported by the National Natural Science Foundation of China
(No.\ 12122111). The research of Chenlin Gu is supported by the National Natural Science Foundation of China
(No.\ 12595280, 12595284).
The research
of Jian Wang is supported by the National Natural Science Foundation of China the National Key R\&D Program of China (2022YFA1006003) and the National Natural Science
Foundation of China (Nos. 12225104 and 12531007).

\vskip 0.3truein
{\small
{\bf Xin Chen:}
   School of Mathematical Sciences, Shanghai Jiao Tong University, 200240 Shanghai, P.R. China. \\
   \texttt{chenxin217@sjtu.edu.cn}

\bigskip

{\bf Chenli Gu:}
   Yau Mathematical Sciences Center, Tsinghua University, Beijing, P.R. China. \texttt{gclmath@tsinghua.edu.cn}

\bigskip

{\bf Jian Wang:}
    School of Mathematics and Statistics \& Key Laboratory of Analytical Mathematics and Applications (Ministry of Education) \& Fujian Provincial Key Laboratory
of Statistics and Artificial Intelligence, Fujian Normal University, 350007 Fuzhou, P.R. China. \texttt{jianwang@fjnu.edu.cn}

\end{document}